\documentclass[11pt]{amsart}

\usepackage{amssymb,amsfonts}
\usepackage{mdwlist}

\usepackage{amssymb,textcomp}
\usepackage{enumerate}

\usepackage[utf8]{inputenc}
\usepackage[T1]{fontenc}
\usepackage{graphicx}
\usepackage[mathscr]{eucal}
\usepackage{bbm}

\usepackage{hyperref}
 \hypersetup{
     colorlinks=true,
     linktocpage=true,
     linkcolor=red,
     filecolor=blue,
     citecolor = blue,
     urlcolor=cyan,
     }
\usepackage[capitalize]{cleveref}
\usepackage{xcolor}

\usepackage[a4paper, twoside=false, vmargin={2cm,3cm}, includehead]{geometry}

\usepackage{comment}
\numberwithin{equation}{section}
\newtheorem{lemma}{Lemma}[section]
\newtheorem{theorem}[lemma]{Theorem}
\newtheorem*{theorem*}{Theorem}

\newtheorem{corollary}[lemma]{Corollary}
\newtheorem{question}[lemma]{Question}
\newtheorem{proposition}[lemma]{Proposition}
\newtheorem*{proposition*}{Proposition}

\newtheorem*{problem*}{Problem}

\newtheorem{theoremL}{Theorem}
\newtheorem{definition}[lemma]{Definition}
\newtheorem{remark}{Remark}

\newcommand{\setN}{{\mathbb{N}}}

\newcommand{\Cov}{\mathrm{Cov}}

\newcommand{\N}{{\mathbb N}}
\renewcommand{\P}{{\mathbb P}}

\newcommand{\Var}{\operatorname{Var}}

\title[]{Ergodic averages for commutative transformations along return times}

\author{Sebasti\'an Donoso}
\address{Departamento de Ingenier\'{\i}a Matem\'atica and Centro de Modelamiento Matem{\'a}tico, Universidad de Chile \& IRL 2807 - CNRS, Beauchef 851, Santiago, Chile}
\email{sdonosof@uchile.cl}

\author{Sovanlal Mondal}
\address{Department of Mathematics, The Ohio State University, Columbus, OH, USA}
\email{mondal.56@osu.edu}

\author{Vicente Saavedra-Araya}
\address{Department of Mathematics, University of Warwick, Coventry, United Kingdom}	
\email{vicente.saavedra-araya@warwick.ac.uk}

\thanks{ The first author was partially funded by ANID/Fondecyt/1241346 and Centro de Modelamiento Matemático (CMM) FB210005, BASAL funds for centers of excellence from ANID-Chile.}


\begin{document}
\begin{abstract}
In this paper, we extend recent results on the convergence of ergodic averages along sequences
generated by return times to shrinking targets in rapidly mixing systems, partially answering questions posed by the first author, Maass and the third author. In particular, for a fixed parameter $a\in (0,1)$ and for \emph{generic} $y\in [0,1]$, we establish both $L^2$ and pointwise convergence for single averages and multiple averages for commuting transformations along the sequences $(a_n(y))_{n\in \setN}$,  obtained by arranging the set
$$\Big\{n\in\N: 0<2^ny \mod{1}<n^{-a} \Big\}$$
in an increasing order.
We also obtain new results for semi-random ergodic averages along sequences of similar type.
\end{abstract}
\maketitle

\section{Introduction}
A {\em measure preserving system}  (m.p.s. for short) is a tuple $(X,\mathcal{X},\mu,T)$  where  $(X,\mathcal{X},\mu)$ is a probability space and $T\colon X\to X$ is a {\em measure preserving transformation}. That is, $T$ is measurable with respect to $\mathcal{X}$ and $\mu(T^{-1}A)=\mu(A)$ for all $A\in \mathcal{X}$. The m.p.s. $(X,\mathcal{X},\mu,T)$ is ergodic if every set in the $\sigma$-algebra $\mathcal{I}(T):=\{A\in \mathcal{X}:\ \mu(A\Delta T^{-1}A)=0\}$ has either zero or full measure.

Two fundamental theorems in ergodic theory are von Neumann's mean ergodic theorem and Birkhoff's pointwise ergodic theorem. The mean ergodic theorem asserts that for any measure-preserving system $(X,\mathcal{X},\mu,T)$ and for any function $f\in L^2$, the {\em time averages} $(1/N)\sum_{n=1}^N f(T^n x)$ converge in $L^2$ to $\mathbb{E}(f|\mathcal{I}(T))$.
 If the system is ergodic, then the above limit becomes the corresponding \emph{space average} $\int_{X} f d\mu$. The pointwise ergodic theorem upgrades the mode of convergence to the pointwise ($\mu$-almost everywhere) setting. 
Following Furstenberg’s ergodic-theoretic proof of Szemerédi’s theorem \cite{Furstenberg_ergodic_szemeredi:1977} in the late 1970s, a major line of research in ergodic theory has focused on the study of {\em multiple ergodic averages}. These expressions generalize the classical von Neumann and Birkhoff ergodic theorems to averages of the form: $$\frac{1}{N} \sum_{n=1}^{N} \prod_{j=1}^k f_j(T_i^{a_j(n)}x)$$
where $(X, \mathcal{X}, \mu, T_i)$ is a m.p.s., the transformations $T_i$ commute, and $f_1, \dots, f_k$ are bounded measurable functions. Regarding the $L^2$-convergence of these averages, we refer the reader to the book by Host and Kra \cite{Host_Kra_nilpotent_structures_ergodic_theory:2018} for a comprehensive overview of the area, as well as Frantzikinakis' open problems in the field \cite{Frantzikinakis_open_problems:2016} (and their updates on Frantzikinakis' website). Regarding pointwise convergence, recent substantial progress was made by several authors using harmonic analysis \cite{kosz2025multilinearcirclemethodquestion,Krause_Mirek_Tao_ptwise_bilinear_ergodic:2022}, particularly regarding polynomial sequences of different degrees. For further details about pointwise convergence of ergodic averages, we refer the reader to the book by Krause \cite{KrauseBook}.

Our main results in this paper are Theorem \ref{thmA}, \ref{thmB}, \ref{thmC} and \ref{thmD}, where we establish convergence of multiple averages along random sequences in different settings. 
The study of ergodic averages along random sequences was first considered by Bourgain in \cite{BO}. Before we state the result, let us introduce some terminology.

\begin{definition}
    Let $(\Omega,\mathcal{F},\mathbb{P})$ be a probability space and let $(X_n)_{n\in \N}$ be a sequence of random variables defined in this space taking values in $\{0,1\}$ such that $\P(X_n=1)=n^{-a}$ for some $a>0$. For $\omega\in\Omega$ such that $X_m(\omega)=1$ for infinitely many $m\in\mathbb{N}$,  we define the \emph{random sequence of integers generated by $(X_n)_{n\in \N}$} as
    \begin{equation}
        a_n(\omega):=\inf \Big\{k\in \N:\ X_1(\omega)+\cdots+X_k(\omega)=n\Big\}.\label{eq:random_sequence}
    \end{equation}\label{def:random_sequence}
\end{definition}
\begin{definition} Let $p\geq 1$. A sequence of integers $(a_n)_{n\in\N}$ is said to be \emph{pointwise  $L^p$-good} if for every m.p.s.  $(X,\mathcal{X},\mu,T)$ and every function $f\in L^p$,
the ergodic averages $$\dfrac{1}{N}\sum_{n=1}^N f(T^{a_n}x)$$ converges for $\mu$-almost every $x\in X$.
    
\end{definition}

Under this terminology, Bourgain in \cite{BO} proved the following result.
\begin{theorem} Let $(\Omega,\mathcal{F},\mathbb{P})$ be a probability space, and let $(X_n)_{n\in\N}$ be a sequence of independent random variables, taking values in $\{0,1\}$, such that $\mathbb{P}(X_n=1)=n^{-a}$ for some $a\in (0,1)$. Then, if $p>1$, for $\mathbb{P}$-almost every $\omega\in \Omega$, the random sequence of integers $(a_n(\omega))_{n\in\N}$ is pointwise $L^p$-good. \label{random_indep}
\end{theorem}
It is known from \cite[Theorem C]{MR1721622} that the above theorem does not hold when $a\geq 1$. See also \cite{MRW} for a finer result in this direction. From now onwards, we will confine our discussion to $a\in (0,1)$.

In \cite{LaVictoire}, LaVictoire extended the above theorem to $p=1$ if $a\in (0,1/2).$  Due to the Law of Large Numbers, for $\mathbb{P}$-almost every $\omega\in \Omega$, $a_n(\omega)/n^{1/(1-a)}$ converges to a non-zero constant. Hence, the random sequence of integers  $(a_n)_{n\in \N}$ can be understood as a random version of $\lfloor n^c \rfloor$ for $c:=\frac{1}{1-a}.$ It is worth mentioning that the deterministic case has been studied recently by Krause and Sun \cite{KrauseSun}, showing that $\lfloor n^c \rfloor$  is pointwise $L^1$-good whenever $c\in (1,8/7)$, extending previous work of Urban and Zienkiewicz \cite{UrbanZienkiewicz} and Mirek \cite{Mirek}. 

Motivated by Theorem \ref{random_indep} and  Return Times Theorem \cite{ReturnTimes}, the first author, Maass, and the third author studied in \cite{Donoso_Maass_Saavedra-Araya_ergodic_return_mixing:2025} random sequences of integers generated by non-independent random variables. As a consequence, they showed the following result.
    \begin{theorem}[Cf. Corollary A in \cite{Donoso_Maass_Saavedra-Araya_ergodic_return_mixing:2025}]
Let $a\in (0,1/2)$ be a real number and $r\geq 2$ be a positive integer. Suppose that $\left(a_n(y)\right)_{n\in \setN}$ is the sequence obtained by arranging the set
\begin{equation}
\left\{n\in \setN: r^n y \mod 1 \in (0,n^{-a})\right\},\label{eqn:DMS_CorollaryA}
\end{equation}
in an increasing order. Then $\left(a_n(y)\right)_{n\in \setN}$
is pointwise $L^2$-good for $\lambda$-almost every $y\in [0,1]$.\label{DMS_CorollaryA}
\end{theorem}
Here and throughout the paper, $\lambda$ denotes the Lebesgue measure on $[0,1]$.  Note that the sequence defined in (\ref{eqn:DMS_CorollaryA}) corresponds to the random sequence of integers generated by the non-independent random variables $$X_n(y):=\mathbbm{1}_{(0,n^{-a})}(r^nx\mod 1).$$
Recently, an extension of Theorem \ref{DMS_CorollaryA} was provided  by the third author in \cite{saavedraaraya2025hittingtimesshrinkingtargets}, replacing the sequence $r^n$ in (\ref{eqn:DMS_CorollaryA}) by more general sequences of integers. 

As first result of this paper, we extend Theorem \ref{DMS_CorollaryA} by showing that the result holds for any $a\in (0,1)$,  partially answering \cite[Question 1]{Donoso_Maass_Saavedra-Araya_ergodic_return_mixing:2025}.

\begin{theoremL}\label{thmA}\label{thm:A}
Let $a\in (0,1)$ be a real number and $r\geq 2$ be a positive integer. Suppose that $\left(a_n(y)\right)_{n\in \setN}$ is the sequence obtained by arranging the set
\begin{equation}
\left\{n\in \setN: r^n y \mod 1 \in \left(0,n^{-a}\right)\right\}
\end{equation}
in an increasing order. Then $\left(a_n(y)\right)_{n\in \setN}$
is pointwise $L^2$-good for $\lambda$-almost every $y\in [0,1]$.
\end{theoremL}

Regarding multiple averages,  the random sequences framework has proven to be a powerful approach, providing pointwise results in cases where the corresponding deterministic setting remains out of reach. An interesting result in this direction was provided by Frantzikinakis, Lesigne and Wierdl in \cite{FLW}.

\begin{theorem} Let $(\Omega,\mathcal{F},\mathbb{P})$ be a probability space, let $(X_n)_{n\in\N}$ be a sequence of independent random variables, taking values in $\{0,1\}$, such that $\mathbb{P}(X_n=1)=n^{-a}$ for some $a\in (0,1/2)$, and let $(a_n)_{n\in\N}$ be the random sequence of integers generated by $(X_n)_{n\in\N}$. 
Then, for $\mathbb{P}$-almost every $\omega\in \Omega$, the following holds: For every probability space $(X,\mathcal{X},\mu)$, every pair of commuting measure-preserving transformations $T_1,T_2:X\to X$, and every $f_1,f_2\in L^{\infty}(\mu)$, the averages
    \begin{equation}
        \dfrac{1}{N}\sum_{n=1}^N f_1 \circ T_1^{a_n(\omega)}\cdot f_2\circ T_2^{a_n(\omega)}\label{FLW_double}
    \end{equation}
    converge in $L^2(\mu)$, and their limit coincides with the $L^2(\mu)$-limit of
\begin{equation*}
\frac{1}{N}\sum_{n=1}^N
f_1\circ T_1^{n} \cdot f_2\circ T_2^{n}.
\end{equation*}
     If $T_1$ and $T_2$ are powers of the same transformation, the convergence can be improved to be pointwise.\label{FLW_Szemeredi}
\end{theorem}

This result was later extended to several transformations by Frantzikinakis, Lesigne and Wierdl in \cite{Frantzikinakis_Lesigne_Wierdl_Szemeredi}.
In this paper, we answer affirmatively the open question 5 posed by the first author, Maass and the third author in \cite{Donoso_Maass_Saavedra-Araya_ergodic_return_mixing:2025}, obtaining an analogue of Theorem \ref{thmA} for averages of the form (\ref{FLW_double}).

\begin{theoremL}\label{thmB}
Let $a\in (0,1/2)$, and let $(a_n(y))_{n\in\mathbb{N}}$ be the random sequence defined in
\eqref{eqn:DMS_CorollaryA}. Then, for $\lambda$-almost every $y\in[0,1]$, the following holds:

For every probability space $(X,\mathcal{X},\mu)$, every pair of commuting
measure-preserving transformations $T_1,T_2\colon X\to X$, and every
$f_1,f_2\in L^{\infty}(\mu)$, the averages
\begin{equation}
\frac{1}{N}\sum_{n=1}^N
f_1\circ T_1^{a_n(y)} \cdot f_2\circ T_2^{a_n(y)}
\end{equation}
converge in $L^2(\mu)$, and their limit coincides with the $L^2(\mu)$-limit of
\begin{equation}\label{avgs}
\frac{1}{N}\sum_{n=1}^N
f_1\circ T_1^{n} \cdot f_2\circ T_2^{n}.
\end{equation}
Moreover, if $T_1$ and $T_2$ are powers of the same transformation, the convergence can be improved to be pointwise.
\end{theoremL}
Note that the existence of the $L^2$-limit of \eqref{avgs} was established in
\cite{Conze_Lesigne}, and pointwise convergence in the case where $T_1$ and $T_2$
are powers of the same transformation was proved in
\cite{Bourgain_double_recurrence_ae_conv:1990}. As an immediate application of the
multiple recurrence theorem together with Furstenberg’s correspondence principle
\cite{Furstenberg_ergodic_szemeredi:1977}, we obtain the following combinatorial
consequence.

\begin{corollary}
Let $a\in(0,1/2)$. For $\lambda$-almost every $y\in[0,1]$, the following holds:
for every set $A\subseteq\mathbb{N}$ with
\[
\overline{d}(A)
:= \limsup_{N\to\infty} \frac{|A\cap\{1,\dots,N\}|}{N} > 0,
\]
there exist $m,n\in\mathbb{N}$ such that
\[
\{m,m+n,m+2n\}\subseteq A
\quad\text{and}\quad
0 < 2^n y \!\!\mod 1 < n^{-a}.\label{Cor}
\]
\end{corollary}

A natural question is whether the range of the parameter in Theorem~\ref{thmB}, and consequently in Corollary \ref{Cor}, can be extended. 
\begin{question}
    Does Theorem \ref{thmB} hold for any $a\in (0,1)?$
\end{question}

Other type of multiple averages studied by Frantzikinakis, Lesigne and Wierdl are the so-called \emph{semi-random ergodic averages}.

\begin{theorem}[\cite{FLW}] Let $(\Omega,\mathcal{F},\mathbb{P})$ be a probability space, let $(X_n)_{n\in\N}$ be a sequence of independent random variables, taking values in $\{0,1\}$, such that $\mathbb{P}(X_n=1)=n^{-a}$ for some $a\in (0,1/14)$, and let $(a_n)_{n\in\N}$ be the random sequence of integers generated by $(X_n)_{n\in\N}$. Then, for $\mathbb{P}$-almost every $\omega\in \Omega$, the following holds: For every probability space $(X,\mathcal{X},\mu)$, every pair of commuting measure-preserving transformations $T_1,T_2:X\to X$, and every $f_1,f_2\in L^{\infty}(\mu)$, 
    $$\lim_{N\to\infty}\dfrac{1}{N}\sum_{n=1}^N f_1(T_1^nx)f_2(T_2^{a_n(\omega)}x)=\mathbb{E}(f_1|\mathcal{I}(T_1))(x)\mathbb{E}(f_2|\mathcal{I}(T_2))(x)$$
    for $\mu$-almost every $x\in X$.
\end{theorem}
This result was extended in \cite{Frantzikinakis_Lesigne_Wierdl_Szemeredi} to $a\in (0,1/2)$ when $T_1=T_2$.

A natural question is whether a similar result can be obtained when the random variables $(X_n)_{n\in\N}$ satisfy a weaker notion of independence. In particular, we focus in the case where the random variables arise as hitting times in rapidly mixing systems, which was the context studied in \cite{Donoso_Maass_Saavedra-Araya_ergodic_return_mixing:2025} for single averages. Before stating our results, we introduce some terminology.

For a function $f\colon [0,1]\subseteq \mathbb{R}\to \mathbb{R}$, recall that 
 $$\|f\|_{\operatorname{BV}}:=\displaystyle \|f\|_{\infty}+\sup_{0=x_0<x_1<\cdots<x_m=1}\sum_{i=0}^{m-1}|f(x_{i+1})-f(x_i)|,$$ where the supremum is taken over all the finite sequences $0=x_0<x_1<\cdots<x_m=1$ of $[0,1]$. We say that $f$ has bounded variation if $\|f\|_{BV}<\infty$. 
\begin{definition}
Let $Y:=([0,1],\mathcal{B}([0,1]),\nu,S)$ be a m.p.s. We say that the system $Y$ has a \emph{decay of correlation for functions of bounded variation  against $L^1(\nu)$ with rate $\rho\colon\mathbb{N}\to \mathbb{R}_{+}$} if there exists $C>0$ such that
    \[\left|\operatorname{Cov}( f,g\circ S^m)\right|=\left|\int_{[0,1]} f\cdot g\circ S^m d\nu-\int_{[0,1]}f d\nu \int_{[0,1]} g d\nu \right|\leq C\rho(m)\|f\|_{\operatorname{BV}}\|g\|_{L^{1}}\]
    for any $m\in \mathbb{N}$, $f$ with bounded variation and $g\in L^1(\nu)$. \label{def:correlations}
\end{definition}
If the system $([0,1],\mathcal{B}([0,1]),\nu,S)$ has a fast decay of correlation, it was shown in \cite[Theorem A]{Donoso_Maass_Saavedra-Araya_ergodic_return_mixing:2025} that, for $\nu$-almost every $y\in [0,1]$, the set $$\{n\in \N:\ S^n y \in (0,n^{-a})\}$$ is pointwise $L^2$-good. In this direction, we obtain the following semi-random ergodic theorem.

\begin{theoremL}
    \label{thmC}
    Let $([0,1],\mathcal{B}([0,1]),\nu,S)$ be a measure-preserving system with decay of correlation for functions with bounded variation against $L^1$ with rate $\rho$. Let $I_n\subseteq [0,1]$ be  a sequence of intervals such that $\nu(I_n)=n^{-a}$ for some $a\in (0,1/14)$, let $\varphi:\mathbb{N}\to \mathbb{N}$ be an increasing function, and suppose that there exists $C>0$ such that \begin{equation}
\rho\Big(\varphi(n+1)-\varphi(n)\Big)\leq \dfrac{C}{n}. \label{thmB_condition}   \end{equation}
    We consider the sequence $(a_n(y))_{n\in\N}$ such that
\[\{n\in\mathbb{N}:\ S^{\varphi(n)}y\in I_n\}=\{a_1(y)<a_2(y)<\cdots\},\]
which is well-defined for $\nu$-almost every $y\in [0,1]$.

Then, for $\nu$-almost every $y\in [0,1]$ the following holds: For every probability space $(X,\mathcal{X},\mu)$, every pair of commuting measure-preserving transformations $T_1,T_2\colon X\to X$, and every $f_1,f_2\in L^{\infty}(\mu)$, 
$$\lim_{N\to\infty}\dfrac{1}{N}\sum_{n=1}^N f_1(T_1^nx)f_2(T_2^{a_n(y)}x)=\mathbb{E}(f_1|\mathcal{I}(T_1))(x)\mathbb{E}(f_2|\mathcal{I}(T_2))(x)$$
for $\mu$-almost every $x\in X$.\label{Semi_random}
\end{theoremL}
\begin{remark}
The condition that each \(I_n\) is an interval can be relaxed to the existence of \(\ell\in\mathbb{N}\) such that each \(I_n\) is a union of at most \(\ell\) intervals. Moreover, with some technical work, condition~\eqref{thmB_condition} can be improved by replacing \(C/n\) with \(C/n^{1-4a}\). However, we will omit the details for the sake of simplicity.
\end{remark}

Note that the sequence $(a_n(y))_{n\in\N}$ in Theorem \ref{Semi_random} is equivalent to the random sequence generated by the non-independent random variables $$X_n(y):=\mathbbm{1}_{I_n}(S^{\varphi(n)}y).$$

Recalling that the doubling map defines a system with exponentially fast decay of correlation for functions with bounded variation against $L^1$, we obtain the following consequence. 
\begin{corollary}
    Let $a\in (0,1/14)$ and $c>1$. For $y\in [0,1],$ we consider the sequence
    $$\Big\{a_1(y)<a_2(y)<\cdots\Big\}=\Big\{n\in\N:\ 0<2^{\lfloor n^c \rfloor}y\mod 1<n^{-a}\Big\}.$$
    Then, for $\lambda$-almost every $y$, the following holds: For every probability space $(X,\mathcal{X},\mu)$, every pair of commuting measure-preserving transformations $T_1,T_2:X\to X$, and every $f_1,f_2\in L^{\infty}(\mu)$,
$$\lim_{N\to\infty}\dfrac{1}{N}\sum_{n=1}^N f_1(T_1^nx)f_2(T_2^{a_n(y)}x)=\mathbb{E}(f_1|\mathcal{I}(T_1))(x)\mathbb{E}(f_2|\mathcal{I}(T_2))(x)$$\label{CorA}
for $\mu$-almost every $x\in X$.
\end{corollary}
Further examples of systems exhibiting a exponentially fast decay of correlation (for a measure absolutely continuous with respect to $\lambda$) includes the Gauss map with its invariant measure\footnote{The measure $\nu \ll \lambda$ defined by $d\nu=\dfrac{1}{\ln(2)\cdot(1+x)}d\lambda$ is invariant for the Gauss map $G$.}, and piecewise expanding maps \cite{Liverani}.

Despite the fact that we are unable to cover the case \(c=1\) with the current method of proof, we believe that a semi-random analogue of Theorem~\ref{thmA} and Theorem~\ref{thmB} should hold. Hence, we pose the following question.
\begin{question}
Does Corollary~\ref{CorA} hold when \(c=1\) and \(a\in(0,1/2)\)?
\end{question}

In the context of random ergodic averages, the behavior of averages of the form
\begin{equation}
    \frac{1}{N}\sum_{n=1}^N f_1\circ T_1^{a_n(\omega)} \cdot f_2\circ T_2^{b_n(\omega)},
\end{equation}
where \((a_n)_n\) and \((b_n)_n\) are independently generated random sequences, is largely unknown, and even simple cases are not understood, both in the $L^2$ and pointwise setting (see Problem~33 in \cite{Frantzikinakis_open_problems:2016}). Although this problem remains out of reach, we tackle a version in which the sequence \(b_n\) is a power of \(a_n\).

\begin{theoremL}\label{thmD} Let $(\Omega,\mathcal{F},\mathbb{P})$ be a probability space, let $(X_n)_{n\in\N}$ be a sequence of independent random variables, taking values in $\{0,1\}$, such that $\mathbb{P}(X_n=1)=n^{-a}$ for some $a\in (0,1/2)$, and let $(a_n(\omega))_{n\in\N}$ be the random sequence of integers generated by $(X_n)_{n\in\N}$ as defined in \eqref{eq:random_sequence}. Let $b>0$ such that $b+2a<1/2$. 
Then, for $\mathbb{P}$-almost every $\omega \in \Omega$, the following holds: For every measure preserving system $(X,\mathcal{X},\mu,T)$, and every $f,g\in L^{\infty}(\mu)$, the averages
\begin{equation}
\frac{1}{N}\sum_{n=1}^N f(T^{a_n(\omega)}x)g(T^{\lfloor a_n(\omega)^{1+b}\rfloor}x) \text{ converge in $L^2$-norm }.\label{thmDequation}
\end{equation}
\end{theoremL}

\subsection*{Organization of the paper}
The paper is organized as follows. In \cref{sec:Preliminaries} we gather the necessary background and technical tools, including the lacunary trick and concentration inequalities for random variables. \cref{sec:ProofsA-B} is dedicated to the proofs of Theorem \ref{thmA} and Theorem \ref{thmB}; we begin by outlining the heuristic strategy and constructing the necessary approximations of return times by independent random variables. In \cref{sec:ProofC}, we establish the semi-random ergodic theorem for systems with decay of correlations, proving Theorem \ref{thmC}. Finally, \cref{sec:ProofD} is devoted to the proof of Theorem \ref{thmD}.

\section{Preliminaries} \label{sec:Preliminaries}

The proof of our results follows the strategy of Frantzikinakis, Wierdl, and Lesigne~\cite{FLW,Frantzikinakis_Lesigne_Wierdl_Szemeredi} for random sequences constructed from independent random variables, although in our setting we must adapt the approach to account for the lack of independence.

A key ingredient in the proof is the so-called \emph{lacunary trick}. This widely used strategy relies on the fact that convergence of averages along suitably chosen lacunary subsequences is sufficient to establish convergence of the corresponding ergodic averages. This is formalized in the following result, whose proof can be found in
\cite[Corollary A.2]{FLW}.

\begin{lemma} Let $(\Omega,\mathcal{F},\mathbb{P})$ be a probability space. For every $n\in \N$, consider a non-negative measurable function  $f_n\colon\Omega\to \mathbb{R}$, and let $(W_n)_{n\in \mathbb{N}}$ be an increasing sequence of positive numbers such that
\begin{equation}
    \lim_{\gamma\to 1^{+}}\limsup_{n\to \infty} \dfrac{W_{\lfloor\gamma^{n+1}\rfloor}}{W_{\lfloor\gamma^{n}\rfloor}}=1.\label{lacunary_trick_eqn}
\end{equation}
For $N\in \N$ and $\omega\in \Omega$, define
\[A_N(\omega):=\dfrac{1}{W_N}\sum_{n=1}^{N} f_n(\omega).\]
Assume that there exist a function $f\colon \Omega\to \mathbb{R}$ and a real-valued sequence $(\gamma_k)_{k\in \mathbb{N}}$ such that $1<\gamma_k<+\infty$, $\gamma_k\to 1$ as $k\to \infty$, and for all $k\in \mathbb{N}$   
\[\lim_{N\to \infty}A_{\lfloor\gamma_k^{N}\rfloor}(\omega)=f(\omega)\]
for $\mathbb{P}$-almost every $\omega\in\Omega$.
Then,
\[\lim_{N\to \infty} A_N(\omega)=f(\omega) \]
for $\mathbb{P}$-almost every $\omega\in\Omega$.\label{lacunary_trick}
\end{lemma}
If we consider a random sequence $(a_n)_{n\in\mathbb{N}}$ generated by $\{0,1\}$-valued random
variables $(X_n)_{n\in\mathbb{N}}$ (see Definition~\ref{def:random_sequence}), then the averages
\begin{equation}
   \frac{1}{N}\sum_{n=1}^N f\bigl(T^{a_n(\omega)}x\bigr)
\quad\text{and}\quad
\frac{1}{\sum_{n=1}^N X_n(\omega)}\sum_{n=1}^N X_n(\omega)\,f(T^n x)\label{equivalence} 
\end{equation}

have the same limiting behavior. A similar equivalence holds for the other types of averages considered in this paper. 

We use the following lemma, proved in \cite[Lemma~2.3]{Donoso_Maass_Saavedra-Araya_ergodic_return_mixing:2025},
to understand the behavior of $\sum_{n=1}^N X_n(\omega)$ in our non-independent setting and, in particular, to show that the
sequence $(a_n(\omega))_{n\in\mathbb{N}}$ is well defined almost surely.

\begin{lemma}
Let $(\Omega,\mathcal{F},\mathbb{P})$ be a probability space and $(X_n)_{n\in \N}$ be a sequence of uniformly bounded and non-negative random variables defined on it. Suppose that there exist $a\in (0,1)$ and $c>0$ such that $\mathbb{E}(X_n)=cn^{-a}$ and $\varepsilon>0$ such that
    \begin{equation*}
    \sum_{n=1}^N\sum_{m=n+1}^N\Cov(X_n,X_m)\ll N^{2-2a-\varepsilon}.
    \end{equation*}
    Then,
    \[\lim_{N\to \infty}\dfrac{1}{W_N}\sum_{n=1}^NX_n(\omega)=1\]
    for almost every $\omega\in \Omega$, where $W_N:=\sum_{n=1}^N \mathbb{E}(X_n)$.\label{LLN}
\end{lemma}
Hence, the study of the limiting behavior of~\eqref{equivalence} can be reduced to understanding
\[
\frac{1}{W_N}\sum_{n=1}^N X_n(\omega)\,f(T^n x),
\]
where $W_N := \sum_{n=1}^N \mathbb{E}(X_n)$. Throughout this paper, $W_N$ behaves like $N^{1-a}$ for
some $a\in(0,1)$. Therefore, we may take advantage of Lemma~\ref{lacunary_trick} to reduce the
problem to the study of lacunary subsequences.

Another classical tool we use is the following version of the \emph{van der Corput trick}. 

\begin{lemma}\label{vdc}
Let $V$ be an inner product space, let $N\in\mathbb{N}$, and let $v_1,\dots,v_N\in V$.
Then, for every integer $M$ with $1\leq M\leq N$, we have
\[
\left\| \sum_{n=1}^{N} v_n \right\|^2
\leq
\frac{2N}{M} \sum_{n=1}^{N}\|v_n\|^2
+
\frac{4N}{M}\sum_{m=1}^{M}
\left|
\sum_{n=1}^{N-m} \langle v_{n+m},v_n \rangle
\right|.
\]
\end{lemma}
We now introduce some notation which will be used throughout the paper.
\subsection*{Notation}
\begin{itemize}
\item 
Let $(a_n)_{n\in\mathbb{N}}$ and $(b_n)_{n\in\mathbb{N}}$ be sequences of positive real numbers.
We write $a_n \ll b_n$ if and only if there exists a constant $C>0$, independent of $n$, such that
$a_n \leq C b_n$ for all $n\in\mathbb{N}$.
We also write $a_n \sim b_n$ if $a_n \ll b_n$ and $b_n \ll a_n$.

\item For a positive integer $n\in \setN$, 
$[n]:=\{1,2,\cdots, n\}$.
\item $\bigsqcup$ denotes \emph{disjoint} union.

\end{itemize}

Finally, we will use the following lemma which is a minor modification of \cite[Lemma 3.3]{FLW}.  For completeness, we provide a proof here.

\begin{lemma}\label{lem:technical}
Let $\left(s(n)\right)$ be a sequence of positive integers  be such that $s(n)\leq n^2$ and $(Z_{m,n})_{m,n\in \setN}$ be a family of random variables, uniformly bounded by $1$, with mean zero. Assume that $\Lambda\subset \setN$ is such that for every fixed $m\in \setN$ the random variables $\{Z_{m,n}: n\in \Lambda \}$ are independent. Further, let $(\rho_n)$ be a sequence of positive numbers and $N_0\in \setN$ such that
\begin{equation}\label{eq:condition}
\sup_{m\in \setN}(\operatorname{Var}(Z_{m,n}))\leq \rho_n \text{ for all $n\in \setN$ and } \frac{\log N} {\sum_{n=1}^N \rho_n}\leq \dfrac{1}{7} \text{ for all $N\geq N_0$}.
\end{equation}
Then, there exists a universal constant $A$ such that for every $N\geq N_0$ and $m\in \setN$, we have
\begin{equation}
\mathbb{P}\left(M_{m,N}\geq A \sqrt{R_N \log N}\right) \leq \frac{2}{N^4},
\end{equation}
where 
\begin{equation}
M_{m,N}= \sup_{t\in [0,1]} \left| \sum_{n\in[N]\cap \Lambda} Z_{m,n}\cdot e(s(n)t)\right|
\end{equation}
and $R_N=\sum_{n=1}^N \rho_n$.
\end{lemma}

\begin{proof}
Fix $m\geq 1$. It will be sufficient to prove the desired estimate for 
\begin{equation}
M_{m,N}= \sup_{t\in [0,1]} \left| P_{m,N}(t) \right|,
\end{equation}
where
\begin{equation}
  P_{m,N}(t)=  \sum_{[N]\cap \Lambda} Z_{m,n}\cdot \cos(2\pi s(n)t)
\end{equation}
In a similar way, one gets an estimate with $\sin(2\pi s(n)t)$ in place of $\cos(2\pi s(n)t)$.

Since $|Z_{m,n}|\leq 1$ and $\mathbb{E}_\omega (Z_{m,n})=0,$ by \cite[Lemma~1.7]{Tao_Vu}, we have $\mathbb{E}_\omega(e^{\lambda Z_{m,n}})\leq e^{\lambda^2 \text{Var} (Z_{m,n})}$ for all $\lambda\in [-1,1]$. Hence, for every $m\in [N], \lambda\in [-1,1],$ and $t\in [0,1]$, we get that
\begin{equation}\label{eq:variance}
\begin{aligned}
\mathbb{E}_\omega (e^{\lambda P_{m,N}(t)})&= \prod_{n\in [N]\cap\Lambda} \mathbb{E}_\omega ( Z_{m,n}\cdot \cos\left(2\pi s(n)t)\right)\\
&\leq \prod_{n\in [N]\cap\Lambda} e^{(\lambda\cos 2\pi s(n)t)^2 \text{Var} (Z_{m,n})}\\
&\leq e^{\lambda^2 R_N},
\end{aligned}
\end{equation}
Since $s(n)\leq n^2$, by using the argument of \cite[Chapter 5, Proposition 5]{Kahane}, there exist random intervals $I_{m,N}$ of length $|I_{m,N}|\geq \frac{1}{N^3}$ such that $|P_{m,N}(t)|\geq \frac{M_{m,N}}{2}$ for every $t\in I_{m,N}$. Using this we get that
\begin{align*}
\mathbb{E}_\omega (e^{\lambda M_{m,N}(t)}/2) &\le N^3\cdot \mathbb{E}_\omega \left(\int_{I_{m,N}}(e^{\lambda P_{m,N}(t)}+e^{-\lambda P_{m,N}(t)})dt\right)\\
&\le N^3\cdot \mathbb{E}_\omega \left(\int_{[0,1]}(e^{\lambda P_{m,N}(t)}+e^{-\lambda P_{m,N}(t)})dt\right)\\
&= N^3\cdot \int_{[0,1]} \mathbb{E}_\omega \left(e^{\lambda P_{m,N}(t)}+e^{-\lambda P_{m,N}(t)}\right)dt\\
&\leq 2 N^3 e^{\lambda^2 R_N}.
\end{align*}
The last inequality follows from \eqref{eq:variance}.

Therefore, we have
\begin{equation}
\mathbb{E}_\omega (e^{\lambda M_{m,N}/2})\leq 2 N^3 \cdot e^{\lambda^2 R_N}.
\end{equation}
Rewriting the above expression we get
\begin{align}
\mathbb{E}_\omega \left(e^{\lambda /2\left(M_{m,N}-2R_N\lambda-2\log (N^7)\lambda^{-1}\right)}\right)\leq \frac{2}{N^4}.
\end{align}
This implies that
\begin{equation}\label{eq:probability}
\mathbb{P}\left(M_{m,N}\geq 2 R_N \lambda + 14 \log N\cdot \lambda^{-1}\right) \leq \frac{2}{N^4}.
\end{equation}
The above inequality is true for all $\lambda\in [-1,1]$. Note that the function $f(\lambda)=2R_N \lambda +14 \log N\cdot \lambda^{-1}$ attains a minimum $\sqrt{28 R_N \log N}$ at $\lambda=\sqrt{\frac{7\log N}{ R_N}}$. By the given condition, $\sqrt{\frac{7\log  N}{ R_N}}<1$ for all $N\geq N_0$. Plugging in $\lambda=\sqrt{\frac{7\log N}{ R_N}}$ in \eqref{eq:probability}, we get the desired conclusion.
\end{proof}

\begin{corollary}\label{lem:Bourgain}
Let $Y_n$ be a sequence of $\{0,1\}$-valued independent random variables with $\mathbb{P}(Y_n=1)=\tau_n$ and $\mathbb{P}(Y_n=0)=1-\tau_n$. Suppose 
\begin{equation}
Y'_n=Y_n-\tau_n,\quad T(N)=\sum_{n=1}^N \tau(n) \text{ and } P_N(t)= \sum_{n=1}^N Y'_ne(nt).
\end{equation}
 Suppose that $\frac{\log N} {T(N)}\leq \frac{1}{7}$ for $N\geq N_0$. Then there is a universal constant $C>0$ such that for every $N\geq N_0$, we have
\begin{equation}
\mathbb{P}\left(\sup_{t\in [0,1]}|P_N(t)|\geq  C\sqrt{ T(N)\log N}\right)\leq \frac{2}{N^4}.
\end{equation}
\end{corollary}

\begin{proof}
Observe that $Y_n'$ satisfies the hypothesis of the above theorem with $\rho_n= \tau_n$.
\end{proof}

\section{Proof of Theorem \ref{thmA} and \ref{thmB}} \label{sec:ProofsA-B}

\subsection{Heuristic idea of the proofs} Firstly, note that the averages
$$
\frac{1}{N} \sum_{n=1}^N   f_1(T_1^ {a_n(y)}x) f_2(T_2^{a_n(y)}x) \quad \text{and}\quad \frac{1}{\sum_{n=1}^N \sigma_n} \sum_{n=1}^N  X_n(y) f_1(T_1^ {n}x) f_2(T_2^{n}x)
$$
have the same limiting behavior. On the other hand, from \cite{Conze_Lesigne}, we know that $$\dfrac{1}{N}\sum_{n=1}^N f_1(T_1^ {n}x) f_2(T_2^{n}x)$$ 
converges in mean, and the convergence is pointwise if $T_1$ and $T_2$ are powers of the same transformation \cite{Bourgain_double_recurrence_ae_conv:1990}. Therefore, it suffices to show that there exists a subset $\Omega\subset [0,1]$ (independent of any m.p.s.) such that $\lambda(\Omega)=1$ and for every $y\in \Omega$, 
\begin{equation}
    \lim_{N\to\infty}\frac{1}{\sum_{n=1}^N \sigma_n} \sum_{n=1}^N  (X_n(y)-\sigma_n) f_1(T_1^ {n}x) f_2(T_2^{n}x)=0\label{reduction}
\end{equation}
for $\mu$-almost every $x\in X$, where $\sigma_n:=\mathbb{E}(X_n)=n^{-a}$. To show (\ref{reduction}), we will approximate $X_n$ by $X_n'$ in such a way so that
\begin{enumerate}[(i)]
   \item $X_{n_k}'$ are mutually independent random variables whenever $n_k$'s are far away from each other, and
    \item for $\lambda$-almost every $y$, $(X_n(y))_n$ and $(X_n'(y))_n$ generate essentially the same random sequence.
 
\end{enumerate}
Based on these properties, we will partition the given random ergodic average into a small number of pieces so that each piece becomes a random ergodic average generated by \emph{independent} random variables. For the independent case, some techniques are already known for showing convergence (for instance, see Theorem \ref{FLW_Szemeredi}). We will use them to get our desired results.
\subsection{Approximation of the given intervals with suitable dyadic intervals}

 Let $I_n=(0, n^{-a})$. 
 For each $n^{-a}$, let $k$ be the positive integer such that  $r^{-(k+1)}< n^{-a}\leq r^{-{k}}$. Use this to define a $r$-adic rational $\gamma_n$ with denominator at most $r^{{{\frac{2k}{a}}}}$ and such that
\begin{equation}
0<(\gamma_n- n^{-a})\leq \frac{1}{r^{{\frac{2k}{a}}}}.
\end{equation}
 Also, we make sure that the chosen sequence $(\gamma_n)$ is non-increasing.\\
 Define
\begin{equation}\label{eq:random2}
J_n=\Big\{y\in [0,1]: r^n y \mod 1 \in (0,\gamma_n)\Big\}
\end{equation}

In the following lemmas, we will show that $J_n$ have some nice properties.

\begin{lemma}\label{lem:prop1}
Let $N\geq 4$ be an integer. For any positive integer $k\geq 2$ and every finite sequence of positive integers $n_1< n_2<n_3<\cdots<n_k\leq N$ satisfying $n_{i}-n_{i-1}>2\log N$ for $ i=2,3,\cdots,k$, we have
\begin{equation}\label{covariance}
\lambda\Big(\bigcap_{i=1}^k J_{n_i}\Big)= \lambda(J_{n_1})\lambda(J_{n_2})\cdots \lambda(J_{n_k}).
\end{equation}
\end{lemma}
\begin{proof} Note that for every $n\in \setN$,
\begin{equation}
J_n= \bigsqcup_{l=0}^{r^n-1} \left(\dfrac{l}{r^n}+\dfrac{1}{r^n}(0,\gamma_n)\right):= \bigsqcup_{l=0}^{r^n-1} J_{l,n}.
\end{equation}

Thus, $J_n$
is a periodic set with period $\frac{1}{r^n}$, and each of its subinterval $J_{l,n}$ has length $\frac{\gamma_n}{r^n}$. Also, for each $l\in \left\{0,1,\cdots,r^{n}-1\right\}$, we define $I_{l,n}$ to be one full period cycle of $J_n$, that is, $I_{l,n}:=\left(\dfrac{l}{r^n},\dfrac{l+1}{r^n}\right)$. It will be sufficient to show that 
\begin{equation}\label{eq:induction}
\lambda\left(I_{l,n_1}\cap \left(\bigcap_{i=1}^k J_{n_i}\right)\right)= \lambda(J_{l,n_1})\cdot \prod_{i=2}^k \lambda(J_{n_{i}}).
\end{equation}
Indeed, summing over $l$ in its range, we will get
\begin{equation}
\lambda\left([0,1]\cap \left(\bigcap_{i=1}^k J_{n_i}\right)\right)= \sum_{l=0}^{r^n-1} \lambda(J_{l,n_1})\cdot \prod_{i=2}^k \lambda(J_{n_{i}})= \prod_{i=1}^k \lambda(J_{n_{i}}).
\end{equation}

We will establish \eqref{eq:induction} by induction on the number of indices. Before that, let us make the following observation. By definition, $\gamma_{n}$ is a rational number of the form $\frac{s}{r^t}$ for some $s,t\in \N$ with $t\leq  2 \log n$. Since for $1\leq i< k-1$, $n_{i+1}>n_i+2\log N\geq n_i+2\log n_i\geq n_i+t$, the length of each subinterval of $J_{n_i}$ is an integer multiple of the period of $J_{n_{i+1}}$. Moreover, the initial point of $J_{l,n_i}$ is the same as the initial point of $I_{s,n_{i+1}}$ for some $s$. Hence, 
\begin{equation}\label{eq:decomposition}
J_{l,n_i}=\bigsqcup_{s\in A_{i+1}} I_{s,n_{i+1}},
\end{equation}
for some indexing set $A_{i+1}$.

Let $n_1,n_2\in\setN$ be such that  $1\leq n_{1}<n_2\le N$ and $n_2-n_1\geq 2\log N$. By \eqref{eq:decomposition}, for any $l\in \left\{0,1,\cdots,r^{n_1}-1\right\}$ we have
\begin{equation}
I_{l,n_1}\cap J_{n_1}\cap J_{n_{2}}=J_{l,n_1}\cap J_{n_{2}}= \bigsqcup_{s\in A_{2}} I_{s,n_{2}} \cap J_{n_{2}}.
\end{equation}
Hence 
\begin{equation}
\begin{split}
    \lambda\left(I_{l,n_1}\cap J_{n_1}\cap J_{n_{2}}\right)=\lambda\left( \bigsqcup_{s\in A_2} I_{s,n_{2}} \cap J_{n_{2}}\right)=\sum_{s\in A_2}\lambda\left( I_{s,n_{2}} \cap J_{n_{2}}\right)\\= \sum_{s\in A_2} \lambda(I_{s,n_{2}})\cdot  \lambda(J_{n_{2}})= \lambda\left( \bigsqcup_{s\in A_2} I_{s,n_{2}}\right)\cdot \lambda(J_{n_{2}})= \lambda(J_{l,n_1})\cdot \lambda(J_{n_{2}}).
\end{split}
\end{equation}

Let $n_1,n_2,\cdots, n_k$ be as in the lemma. Suppose that \eqref{eq:induction} is true for the last $(k-1)$ indices. That is,
\begin{equation}\label{eq:hypo}
\lambda\left(I_{s,n_2}\cap \left(\bigcap_{i=2}^k J_{n_i}\right)\right)= \lambda(J_{s,n_2})\cdot \prod_{i=3}^k \lambda(J_{n_{i}}).
\end{equation}
 By \eqref{eq:decomposition}, for any $l\in \left\{0,1,\cdots r^{n_1}-1\right\}$, we have
\begin{equation}
I_{l,n_1}\cap \left(\bigcap_{i=1}^k J_{n_i}\right)=J_{l,n_1}\cap\left(\bigcap_{i=2}^k J_{n_i}\right)= \bigsqcup_{s\in A_2} I_{s,n_{2}} \cap \left(\bigcap_{i=2}^k J_{n_i}\right).
\end{equation}
Hence 
\begin{equation*}
\begin{split}
\lambda\left(I_{l,n_1}\cap \left(\bigcap_{i=1}^k J_{n_i}\right)\right)=\lambda\left( \bigsqcup_{s\in A_2} I_{s,n_{2}} \cap \left(\bigcap_{i=2}^k J_{n_i}\right)\right)=\sum_{s\in A_2}\lambda\left( I_{s,n_{2}} \cap \left(\bigcap_{i=2}^k J_{n_i}\right)\right)\\= \sum_{s\in A_2} \lambda(J_{s,n_{2}})\cdot \prod_{i=3}^k \lambda(J_{n_{i}}) =  \lambda\left( \bigsqcup_{s\in A_2} I_{s,n_{2}}\right)\cdot \prod_{i=2}^k \lambda(J_{n_{i}}) = \lambda(J_{l,n_1})\cdot \prod_{i=2}^k \lambda(J_{n_{i}}),
\end{split}
\end{equation*}
where we used our induction hypothesis \eqref{eq:hypo} and  the fact that $\lambda(J_{s,n_2})= \lambda(I_{s,n_2})\cdot \lambda(J_{n_2})$.
This proves \eqref{eq:induction} and thereby finishes the proof of the lemma. 
\end{proof}

\begin{lemma}\label{lem:prop2}
Let
\begin{equation}
S(y):=\Big\{n\in \setN: r^n y \mod 1 \in I_n\Big\} \text{ and } T(y):=\Big\{n\in \setN: r^n y \mod 1 \in J_n \Big\}
\end{equation}
Then for $\lambda$-almost every $y\in [0,1]$, $S(y)\triangle T(y)$ is a finite set, where $\triangle$ means symmetric difference.
\end{lemma}

\begin{proof}
Let
\begin{equation}
F_n=\Big\{y\in [0,1]: r^n y \mod 1 \in (n^{-a},\gamma_n)\Big\}.
\end{equation}
Clearly, $\lambda(F_n)=(\gamma_n-n^{-a})$.
By the Borel-Cantelli Lemma, it suffices to show that $\lambda(F_n)$ is summable. Indeed,
\begin{align*}
\sum_{n\in \N}\lambda(F_n)=\sum_{k\in\N}\sum_{n=r^{\frac{k}{a}}}^{r^{\frac{k+1}a}} (n^{-a}-\gamma_n) \leq \sum_{k\in\N}\sum_{n=r^{\frac{k}a}}^{r^{\frac{(k+1)}a}} \frac{1}{r^{\frac{2k+1}{a}}}<\infty,
\end{align*} 
as desired. \end{proof}

\subsection{Proof of Theorem~\ref{thmA}}
Let $X_n'$ be a sequence of $\{0,1\}$-valued random variables defined by
\begin{equation}
X_n' (y)=1 \text{ iff } r^n y \mod 1 \in J_n.
\end{equation}
We have $\lambda(X_n'=1)=\lambda(J_n)=\gamma_n$, and let $Z_n=X_n'-\gamma_n$ and $W_N=\sum_{n=1}^N\gamma_n$. By Lemma~\ref{lem:prop2}, it will be sufficient to show that for Lebesgue a.e. $y$ and for every $f\in L^2(X)$
\begin{equation}\label{eq:toshow}
\frac{1}{W_N} \sum_{n=1}^N  Z_n(y) f(T^n x) \to 0 \text{ for a.e. $x$}.
\end{equation}
By Herglotz's spectral theorem and the lacunary trick (i.e., Lemma~\ref{lacunary_trick}), it will be sufficient to show that there exists a sequence $(b_n)$, and for Lebesgue almost every $y$ there exists a constant $C(y)$ such that
 \begin{equation}\label{eq:expobound1}
\left|\sup_{t} \sum_{n=1}^N Z_n(y)e(nt)\right|\leq C(y) b_N,
 \end{equation}
where $e(x)=e^{2\pi i x}$, and for every $\rho>1$,
\begin{equation}\label{eq:expobound2}
\sum_{N=1}^\infty\left|\dfrac{b_{\lfloor\rho^N \rfloor}}{W_{\lfloor\rho^N\rfloor}}\right|^2<\infty.
\end{equation}

To prove this, we want to apply Corollary~\ref{lem:Bourgain}. So, we write
 \begin{align*}
    \{1,\cdots,N\}=\bigsqcup_{k\in R}S_k ,\text{ where } R=\lceil{2\log N \rceil} \text{ and }
S_k=\{n\in [N]: n\equiv k \mod R\}.
 \end{align*}
 Without loss of any generality, we can assume that $R$ divides $N$. Indeed, if $R$ does not divide $N$, then we write $R=N'+j$ for some $j$, $0\leq j\leq R$. Then we work with $N'$. The contribution of $N'+1,N'+2,\cdots, N$ will be insignificant.
Let $C$ be the constant of Corollary~\ref{lem:Bourgain}. By union bound, we get
\begin{equation}\label{eq:union}
\begin{split}
  &\lambda\left(\sup_t\left|\sum_{n=1}^NZ_n(y)e(nt)\right|\geq C R\sqrt{W_N\log N}\right)\\
  &\leq \lambda\left(\bigcup_{k=1}^R\left\{\sup_t\left|\sum_{n\in S_k}Z_n(y)e(nt)\right|\geq C \sqrt{W_N\log N}\right\}\right)\\ 
  &\leq \sum_{k=1}^R \lambda\left(\sup_t\left|\sum_{n\in S_k}Z_n(y)e(nt)\right|\geq C \sqrt{W_N\log N}\right).  
\end{split}
\end{equation}

For a fixed $k\in [R]$,
\begin{equation}\label{eq:reduction}
\begin{split}
\sup_t\left|\sum_{n\in S_k} Z_n(y)e(nt)\right|&= \sup_t\left|\sum_{q=1}^{N/R } Z_{k+qR}(y)e((k+qR)t)\right|\\
&= \sup_{t}\left|\sum_{q=1}^{N/R } Z'_{q}(y)e(qRt)\right|,
\end{split}
\end{equation}

where $Z_q'(y)=Z_{k+qR}$ when $k+qR\leq N$, otherwise $Z_q'(y)=0$. By Lemma~\ref{lem:prop1}, $Z_q'$ are independent random variables. Observe that $\Var(Z_q')\leq \gamma_{k+qR}$. Hence, applying Corollary~\ref{lem:Bourgain}, we get

\begin{equation*}
\lambda\left(\left\{\sup_t\left|\sum_{q=1}^{N/R} Z'_{q}(y)e(qt')\right|\geq C \sqrt{\left(\sum_{q=1}^{N/R}{\gamma_{k+qR}}\right)\log (N/R)} \right\}\right) \leq \frac{1}{ (N/R)^4}\leq \dfrac{R^4}{N^4}.
\end{equation*}
Then,
\begin{equation}\label{eq:fixk}
\lambda\left(\left\{\sup_t\left|\sum_{q=1}^{N/R} Z'_{q}(y)e(qt')\right|\geq C \sqrt{W_N\log N} \right\}\right) \leq \frac{R^4}{N^4}.
\end{equation}
Combining \eqref{eq:union},\eqref{eq:reduction} and \eqref{eq:fixk}, we get

\begin{equation}
\lambda \left(\left\{\sup_t\left|\sum_{n=1}^{N}Z_n(y)e(nt)\right|\geq C R\sqrt{W_N\log N}\right\}\right)\leq \frac{R^5}{N^4}
\end{equation}

Hence, by applying Borel-Cantelli lemma and using the fact that $R= \lceil 2\log N \rceil$, for Lebesgue a.e. $y$ there is a constant $C(y)$ such that
\begin{align}
\sup_t\left|\sum_{n=1}^N Z_n(y)e(nt)\right| \leq C(y)(2\log N+1) \sqrt{ N^{1-a} \log N}.
\end{align}
Thus, taking $b_n=(2\log N+1) \sqrt{ N^{1-a} \log N}$ both \eqref{eq:expobound1} and \eqref{eq:expobound2} are satisfied. This proves \eqref{eq:toshow} and thereby finishes the proof of Theorem~\ref{thm:A}.

\subsection{Proof of Theorem~\ref{thmB}}
\begin{proof}
We use the notation from the last proof and  without loss of generality assume that $R$ divides $N$.  As before it is sufficient to show that for Lebesgue a.e. $y$ and for every $f\in L^2(X)$
\begin{equation}
\frac{1}{W_N} \sum_{n=1}^N  Z_n(y) f_1(T_1^ {n}x) f_2(T_2^{n}x) \to 0 \text{ for a.e. $x$}.
\end{equation}
By the lacunary trick (i.e., Lemma~\ref{lacunary_trick}) and the Borel-Cantelli lemma, it is enough to show that
\begin{equation}\label{eq:toshow1}
\sum_{N\in \left\{\lfloor\rho^k\rfloor:k\in \setN\right\}} \left\|\frac{1}{N^{1-a}}\sum_{n=1}^N Z_n(y) T_1^ {n}f_1\cdot T_2^{n} f_2 \right\|^2_{L^2(\mu)}<\infty.
\end{equation}
We have that
\begin{align*}
A_N(y)&:=N^{2-2a}\cdot\left\|\sum_{n=1}^N Z_n(y) T_1^ {n}f_1\cdot T_2^{n} f_2 \right\|^2_{L^2(\mu)}\\
&= N^{2-2a}\cdot\left\|\sum_{k=1}^R\sum_{n\in [N]\cap S_k} Z_n(y) T_1^ {n}f_1\cdot T_2^{n} f_2 \right\|^2_{L^2(\mu)}\\
&\leq N^{2-2a}\cdot \left(\sum_{k=1}^R \left\|\sum_{n\in [N]\cap S_k} Z_n(y) T_1^ {n}f_1\cdot T_2^{n} f_2 \right\|_{L^2(\mu)} \right)^2\\
&\leq (2 \log N+1) \cdot \sum_{k\in [R]}A_{N}^{(k)}(y),
\end{align*}

where
\begin{align*}
A_{N}^{(k)}(y):&=N^{2-2a}\cdot\left\|\sum_{n\in [N]\cap S_k} Z_n(y) T_1^ {n}f_1\cdot T_2^{n} f_2 \right\|^2_{L^2(\mu)},
\end{align*}
 and in the last step we used the Cauchy-Schwarz inequality
and the fact that $R= \lceil 2\log N \rceil$.

\begin{align*}
A_{N}^{(k)}(y):&=N^{2-2a}\cdot\left\|\sum_{n\in [N]\cap S_k} Z_n(y) T_1^ {n}f_1\cdot T_2^{n} f_2 \right\|^2_{L^2(\mu)}\\
&= N^{2-2a}\cdot\left\|\sum_{q\in [N/R]} Z_{k+qR}(y) T_1^ {k+qR}f_1\cdot T_2^{k+qR} f_2 \right\|^2_{L^2(\mu)}\\
&=N^{2-2a}\cdot\left\|\sum_{q\in [N/R]} Z'_{q}(y) T_1^ {\phi(q)}f_1\cdot T_2^{\phi(q)} f_2 \right\|^2_{L^2(\mu)},
\end{align*}
where
\begin{equation*}
\phi(q)={k+qR} \text{ and }  Z'_{q}=Z_{k+qR} \text{ when $k+qR\leq N$, otherwise $Z_q'(y)$=0} .
\end{equation*}
Applying van der Corput's lemma with $M=N/R$, we get
$A_{N}(k)=A_{1,N}(k)+A_{2,N}(k)$,
where 
\begin{equation*}
A_{1,N}^{(k)}(y):= N^{2-2a}\cdot\sum_{q\in [N/R]} \left\|Z'_{q}(y) T_1^ {\phi(q)}f_1\cdot T_2^{\phi(q)} f_2 \right\|^2_{L^2(\mu)}
\end{equation*}
and 
\begin{equation*}
A_{2,N}^{(k)}(y):=N^{2-2a}\cdot\sum_{m=1}^M \left|\sum_{n=1}^{M-m} \int Z'_{n+m}(y)\cdot Z'_n(y)\cdot T_1^{\phi(n+m)} f_1\cdot T_2^{\phi(n+m)} f_2\cdot T_1^{\phi(n)} \overline{f_1}\cdot T_2^{\phi(n)} \overline{f_2}\right|
\end{equation*}
Composing $T_1^{-nR}$, using Cauchy-Schwarz inequality, assuming $S=T_1^{-R}T_2^R$ and $g_m=T_2^{k+mR}f_2\cdot T_2^k \overline{f_2}$ we get
\begin{align*}
A_{2,N}^{(k)}(y)&\leq N^{2-2a}\cdot\sum_{m=1}^M \left|\left|\sum_{n=1}^{M-m} Z'_{n+m}(y)\cdot Z'_n(y) S^{n} g_m\right|\right|_{L^2(\mu)}\\
&\leq N^{1-2a}\cdot M_{m,N}(y),
\end{align*}
where
\begin{equation*}
M_{m,N}(y):=\max_{1\leq m\leq M} \sup_{t\in [0,1]} \left|\sum_{n=1}^{M-m} Z'_{n+m}(y)\cdot Z'_n(y) e(nt)\right|
\end{equation*}
Our goal is to apply Lemma~\ref{lem:technical} with taking $Z_{m,n}=Z'_{n+m}\cdot Z'_n $ and $\Lambda=\Lambda_l:= [N-m]\cap \Lambda_{l,m}$, $l=1,2$ where $\Lambda_{l,m}$ is defined as follows:
\begin{equation}\label{eq:Lambda}
\begin{aligned}
\Lambda_{1,m}&=\{n: 2km<n\leq (2k+1)m \text{ for some non-negative integer }k\}, \text{ and }\\
\Lambda_{2,m}&=\{n: (2k+1)m<n\leq (2k+2)m \text{ for some non-negative integer }k\}.
\end{aligned}
\end{equation}
Observe that for any $\eta>0$,
\begin{align*}
\lambda\left(y: M_{m,N}(y)>\eta \right)&\leq \sum_{m\leq M} \lambda\left(y:\sup_{t\in [0,1]} \left|\sum_{n=1}^{M-m} Z'_{n+m}(y)\cdot Z'_n(y) e(nt)\right|>\eta\right)\\
&\leq \sum_{l=1,2} \sum_{m\leq M} \lambda\left(y:\sup_{t\in [0,1]} \left|\sum_{n\in [M-m]\cap \Lambda_l} Z'_{n+m}(y)\cdot Z'_n(y) e(nt)\right|>\eta/2\right)
\end{align*}
In addition, $\Lambda_{1,m}\cup\Lambda_{2,m}=\setN$, and by Lemma~\ref{lem:prop1}, $\{Z_{m,n}: n\in \Lambda_{k,m}\}$ is independent, and 
\begin{align*}
\text{Var}(Z_{m,n})&= \sigma_n \sigma_{m+n}-\sigma_n\sigma_{n+m}^2-\sigma_n^2\sigma_{n+m}+ \sigma_n^2 \sigma_{n+m}^2\\
&\leq \sigma_n\sigma_{n+m}\leq \sigma_n^2\sim n^{-2a}.
\end{align*}

Hence, by Lemma~\ref{lem:technical}, there exists a set $E_{k,N}$ with $\lambda(E_{k,N})\leq \frac{2}{N^4}$ such that for every $y\notin E_k$
\begin{equation*}
M_{m,N}(y)\leq N^{\frac{1}{2}-a}.
\end{equation*}
Let $E_N:=\cup_{k=1}^R E_{k,N}$. Then $\lambda(E_N)\leq \dfrac{5\log N}{N^4}$, and for any $y\notin E_N$, we must have
\begin{equation*}
A_N(y)\leq 2\log^3 N\left(N^{1-2a} + N^{\frac{3}{2}-3a}\right).
\end{equation*}
Since $\sum_{N} \lambda (E_N)<\infty$, by Borel-Cantelli lemma, there exists a set $Y'\subset[0,1]$ of full Lebesgue measure such that $y\notin E_N$ eventually. Since $(1-2b)<0$ and for any $\epsilon>0$ $\lim_{x\to \infty}\dfrac{\log x}{x^\epsilon}=0$, for any $y\in Y'$, \eqref{eq:toshow1} is satisfied. 
\end{proof}

\section{Proof of Theorem  \ref{thmC}} \label{sec:ProofC}
Throughout this section, let $([0,1],\mathcal{B}([0,1]),\nu,S)$ be a measure-preserving system with rate \(\rho\) of decay of correlations between functions of bounded variation and $L^1$ observables (Definition \ref{def:correlations}), and let $(I_n)_{n\in\mathbb{N}}$ be a sequence of intervals in $[0,1]$ with $\nu(I_n)=n^{-a}$ for some $a\in(0,1/14)$. 

Let $\varphi:\mathbb{N}\to\mathbb{N}$ be an increasing function, and assume that
\[
\rho(\Delta_n)\ll \dfrac{1}{n}, \text{ where } \Delta_n:=\varphi(n+1)-\varphi(n).
\]

For each $y\in[0,1]$, we consider the sequence of \emph{$\varphi$-hitting times of $y$ to $(I_n)_n$} as

\[
\{n\in\mathbb{N} : S^{\varphi(n)}y \in I_n\}
   = \{a_1(y) < a_2(y) < \cdots\}.\]

By defining 
\begin{equation}
   X_n(y):=\mathbbm{1}_{I_n}(S^{\varphi(n)}y),\label{def:random_variables} 
\end{equation}
we can write $$\{a_1(y) < a_2(y) < \cdots\}=\{n\in \N:\ X_n(y)=1\}.$$

Despite the lack of independence between the random variables $X_n$, we still can establish a good control on the dependence by taking advantage of the decay of correlation for the system  and the growth rate of $\varphi$. In fact, if $m\geq n$, notice that  
\begin{equation}
\begin{split}    \operatorname{Cov}(X_n,X_m)&=\int_{[0,1]}\Big(\mathbbm{1}_{I_n}\circ S^{\varphi(n)}\Big)\Big(\mathbbm{1}_{I_m}\circ S^{\varphi(m)}\Big)d\nu-\int_{[0,1]} \mathbbm{1}_{I_n}d\nu\int_{[0,1]}\mathbbm{1}_{I_m}d\nu\\
    &=\int_{[0,1]} \mathbbm{1}_{I_n}\cdot \mathbbm{1}_{I_m}\circ S^{\varphi(m)-\varphi(n)}d\nu-\int_{[0,1]} \mathbbm{1}_{I_n}d\nu\int_{[0,1]}\mathbbm{1}_{I_m}d\nu.
\end{split}\label{Covariance_control}
\end{equation}
Since each interval $I_n$ is union of at most $\ell$ intervals, $\mathbbm{1}_{I_n}$ has bounded variation, and moreover, the bound is uniform for all $n$. Thus, recalling that $\rho$ denotes the rate of decay of correlation, for any $m\geq n$ we obtain that 
\begin{equation}
\begin{split}
      |\operatorname{Cov}(X_n,X_m)|&\ll \rho\Big(\varphi(m)-\varphi(n)\Big)\|\mathbbm{1}_{I_n}\|_{BV}\cdot\|\mathbbm{1}_{I_m}\circ S^{\varphi(m)}\|_{L^1(\nu)}\\
      &\ll \rho(\Delta_n)\mathbb{E}(X_m).
\end{split}
\end{equation}
In particular, notice that $$\sum_{n=1}^N\sum_{m=n+1}^N|\operatorname{Cov}(X_n,X_{n+m})|\ll \log(N)N^{1
-a}.$$
By a straightforward application of Lemma 2.3 in \cite{Donoso_Maass_Saavedra-Araya_ergodic_return_mixing:2025}, we obtain the following.
\begin{proposition} Assume the hypotheses of Theorem~\ref{thmC}, and let $(X_n)_{n\in\mathbb{N}}$ be as defined in
\eqref{def:random_variables}. Then, for Lebesgue-almost every $y\in [0,1]$, \begin{equation}
    \lim_{N\to\infty}\frac{1}{N^{1-a}}\sum_{n=1}^NX_n(y)=1.\label{prop:LLN}
\end{equation}\label{prop:LLN2}
\end{proposition}
The main strategy in the proof of Theorem \ref{thmC} is motivated by the semi-random case for independent random variables studied by Frantzikinakis, Wierdl and Lesigne \cite{FLW}. Specifically, we make use of the following reduction.

\begin{proposition}[Cf. Proposition~2.4 in \cite{FLW}]\label{prop:reduction}
Assume the hypotheses of Theorem~\ref{thmC}, and let $(X_n)_{n\in\mathbb{N}}$ be as defined in
\eqref{def:random_variables}.

Suppose that for Lebesgue-almost every $y\in[0,1]$, the following holds: for every probability space
$(X,\mathcal{X},\mu)$, every pair of commuting measure-preserving transformations
$T_1,T_2\colon X\to X$, and every $f_1,f_2\in L^{\infty}(\mu)$, we have
\begin{equation}\label{prop:reduction_equation}
\sum_{k=1}^{\infty}
\left\|
\frac{1}{W_{[\gamma^k]}}
\sum_{n=1}^{[\gamma^k]}
Z_n(y)\cdot T_1^{X_1(y)+\cdots+X_n(y)}f_1\cdot T_2^n f_2
\right\|_{L^2(\mu)}^{2}
<+\infty,
\end{equation}
where $Z_n:=X_n-\sigma_n$.
Then, the conclusion of Theorem~\ref{Semi_random} holds.
\end{proposition}

In this way, the proof of Theorem~\ref{thmC} is reduced to showing that
\eqref{prop:reduction_equation} holds. It is worth noting that \cite[Proposition~2.4]{FLW} was
proved under the assumption that the random variables $(X_n)_{n\in\mathbb{N}}$ are independent and
satisfy $\mathbb{P}(X_n=1)=n^{-a}$ for some $a\in(0,1/14)$.  
We emphasize that, provided condition~(\ref{prop:LLN}) holds, independence is not required in the
original proof, and the same argument applies without modification in our setting; hence we omit
the details.

Considering this reduction, we devote the rest of the section in showing \eqref{prop:reduction_equation}. We begin by establishing bounds for the products $X_{n+m}X_n$.

\begin{proposition}
Assume the hypotheses of Theorem~\ref{thmC}, let $(X_n)_{n\in\mathbb{N}}$ be as defined in
\eqref{def:random_variables}, and let $b\in (a,1/14)$. Then, for almost every $y\in [0,1],$
\[
\sum_{m=1}^{\lfloor N^b\rfloor}\sum_{n=1}^N X_{n+m}(y)X_n(y) \ll N^{b+1-2a}.
\]\label{prop:interaction}
\end{proposition}
\begin{proof}
    Let us consider  \[S_N:=\sum_{m=1}^{\lfloor N^b\rfloor}\sum_{n=1}^{N}(X_{n+m}X_n-\mathbb{E}(X_{n+m})\mathbb{E}(X_n))\]
and
\[A_N:=N^{-c}S_N \ \text{ where }  \ c:=b+1-2a.\]
Since $\mathbb{E}(X_{n+m})\ll \mathbb{E}(X_n)$, we have
\[N^{-c}\sum_{m=1}^{\lfloor N^b\rfloor}\sum_{n=1}^{N}\mathbb{E}(X_{n+m})\mathbb{E}(X_n)\ll N^{-c}\sum_{m=1}^{\lfloor N^b\rfloor}\sum_{n=1}^{N}n^{-2a}\ll 1.\]
So it suffices to show that almost every $y\in [0,1]$, $$ \lim_{N\to \infty} A_N(y)=0.$$

Expanding $S_N^2$, we can write
\begin{align*}
    S_N^2=\sum_{m=1}^{\lfloor N^b\rfloor}\sum_{n=1}^{N}\sum_{m'=1}^{\lfloor N^b\rfloor}\sum_{n'=1}^{N} &\Big(X_{n+m}X_{n}X_{n'+m'}X_{n'}+\sigma_{n+m}\sigma_n\sigma_{n'+m'}\sigma_{n'}\\
    &-X_{n+m}X_n\sigma_{n'+m'}\sigma_{n'}-X_{n'+m'}X_{n'}\sigma_{n+m}\sigma_{n}\Big).
\end{align*}
Taking expectation and rearranging terms, we can write 
\begin{equation}
    \begin{split}
        \mathbb{E}(S_N^2)&\ll \sum_{m=1}^{\lfloor N^b\rfloor}\sum_{n=1}^N\sum_{m'=1}^{\lfloor N^b\rfloor}\sum_{n'=1}^N \Bigg(|\operatorname{Cov}(X_n,X_{n+m}X_{n'}X_{n'+m'})|\\
        &+\sigma_{n'}\sigma_{n'+m'}|\operatorname{Cov}(X_{n},X_{n+m})|
        +\sigma_n|\operatorname{Cov}(X_{n+m},X_{n'}X_{n'+m'})|\Bigg).
    \end{split}
\end{equation}

Let $n'>n$. If we proceed similarly to (\ref{Covariance_control}) and we denote $$h_{n,m,n',m'}=\mathbbm{1}_{I_{n+m}}\circ S^{\varphi(n+m)-\varphi(n+1)}\cdot \mathbbm{1}_{I_{n'+m'}}\circ S^{\varphi(n'+m')-\varphi(n+1)}\cdot \mathbbm{1}_{I_{n'}}\circ S^{\varphi(n')-\varphi(n+1)},$$ we have that
\begin{equation*}
    \begin{split}
        |\Cov(X_n,X_{n+m}X_{n'}X_{n'+m'})|=&\Big|\int_{[0,1]} \mathbbm{1}_{I_n}\cdot h_{n,m,n',m'}\circ S^{\varphi(n+1)-\varphi(n)}d\nu
        \\
        &-\int_{[0,1]}\mathbbm{1}_{I_n}d\nu\int_{[0,1]} h_{n,m,n',m'}d\nu\Big|\\
        &\ll \rho\Big(\varphi(n+1)-\varphi(n)\Big) \mathbb{E}(X_{n'}X_{n+m}X_{n'+n'})\\
        &\ll \frac{1}{n}.
    \end{split}
\end{equation*}
With a similar argument, we obtain
 
$$|\Cov(X_{n},X_{n+m})|\ll \frac{1}{n},$$ and for $n+m<n'$,
$$|\Cov(X_{n+m},X_{n'}X_{n'+m'})|\ll \dfrac{1}{n+m}.$$

Therefore, 
\begin{align*}
    \mathbb{E}(S_N^2)&\ll \sum_{m=1}^{\lfloor N^b\rfloor}\sum_{n=1}^N\sum_{m'=1}^{\lfloor N^b\rfloor}\sum_{n'=n+m+1}^{N} n^{-1}+\sum_{m=1}^{\lfloor N^b\rfloor}\sum_{n=1}^N\sum_{m'=1}^{\lfloor N^b\rfloor}\sum_{n'=n}^{n+m} 1\\
    &\ll N^{1+2b}\log(N)+N^{1+3b}.\\
\end{align*}
Note that $N^{-2c}=N^{-2-2b+4a},$ thus
 $$\mathbb{E}(A_N^2)\ll N^{1+2b-2c}+N^{1+3b-2c}=N^{-1+4a}+N^{-1+b+4a}.$$
Since $-1+4a<0$ and $-1 + b + 4a < 0$ , a classical application of Markov's inequality, the Borel-Cantelli lemma, and the lacunary trick shows that $A_N^2$ converges to $0$ almost surely, and hence $A_N$ also converges to $0$.

\end{proof}

\begin{corollary}
Assume the hypotheses of Theorem~\ref{thmC}, let $(X_n)_{n\in\mathbb{N}}$ be as defined in
\eqref{def:random_variables}, and let $b\in (a,1/14)$. Then, for Lebesgue-almost every $y\in [0,1],$
 \[\sum_{m=1}^{\lfloor N^b\rfloor}\sum_{n=1}^{N} Z^2_{n+m}(y)Z_n^2(y) \ll N^{b+1-2a},\]
 where $Z_n=X_n-\sigma_n.$\label{cor:C_1}
 \end{corollary}
\begin{proof}
By definition, we have
 \[Z_n^2=X_n-2\sigma_nX_n+\sigma_n^2.\]
 Therefore, since $X_n\in \{0,1\}$,
 \begin{align*}
     Z_n^2Z_{n+m}^2&=X_{n}X_{n+m}-2X_nX_{n+m}\sigma_{n+m}+X_n\sigma_{n+m}^2\\
     &-2\sigma_nX_nX_{n+m}+4X_nX_{n+m}\sigma_n\sigma_{n+m}-2X_n\sigma_n\sigma_{n+m}^2\\
     &+X_{n+m}\sigma_n^2-2\sigma_{n}^2\sigma_{n+m}X_{n+m}+\sigma_{n}^2\sigma_{n+m}^2.\\
     &\leq X_nX_{n+m}(1-2\sigma_{n+m}-2\sigma_n+4\sigma_{n}\sigma_{n+m})+\sigma_n^2
 \end{align*}
By noticing that
 \[\sum_{m=1}^{\lfloor N^b\rfloor}\sum_{n=1}^{N}\sigma_n^2\ll N^{b+1-2a},\]
 the conclusion follows from Proposition \ref{prop:interaction}.
 \end{proof}
The last ingredient we need to establish \eqref{prop:reduction_equation} is the following lemma. With this result, Theorem \ref{thmC} follows directly from Proposition \ref{prop:reduction}.

 \begin{lemma}  Let $([0,1],\mathcal{B}([0,1]),\nu,S)$ be a measure-preserving system with rate $\rho$ of decay of correlation for functions with bounded variation against $L^1$.  Let $I_n\subseteq [0,1]$ be  a sequence of intervals such that $\sigma_n:=\nu(I_n)=n^{-a}$ for some $a\in (0,1/14)$. Let $\varphi:\mathbb{N}\to \mathbb{N}$ be an increasing function such that $$\rho\Big(\varphi(n+1)-\varphi(n)\Big)\ll  \dfrac{1}{n},$$ and  $(X_n)_n$ be as defined in (\ref{def:random_variables}).
 
  For almost every $y\in [0,1]$, it holds: For every $\gamma>1$, every probability space $(X,\mathcal{X},\mu)$, every pair of commutative measure-preserving transformations $T_1,T_2:X\to X$ and every $f_1,f_2\in L^{\infty}(\mu)$, we have that 
 \begin{equation}
    \sum_{N=1}^{\infty}\left \| \dfrac{1}{W_{\lfloor\gamma^N\rfloor}}\sum_{n=1}^{\lfloor\gamma^N\rfloor} Z_n(y)\cdot T_1^{X_1(y)\cdots+X_n(y)}f_1\cdot T_2^nf_2\right \|_{L^2}^{2}<+\infty,
\end{equation}
where $Z_n:=X_n-\sigma_n$ and $W_N:=\sum_{n=1}^N \sigma_n$.
 \end{lemma}

 \begin{proof}
Let $\varepsilon \in \Big(0,\frac{1/2-7a}{3}\Big)$, and define $b:=a+\varepsilon$. Applying Lemma (\ref{vdc}) with $v_n=Z_n(y)T_1^{X_1(y)+....+X_n(y)}f_1\cdot T_2^n f_2$ and $M:=\lfloor N^b \rfloor$, we obtain that 
\begin{equation}A_N(y):=\left\| \dfrac{1}{W_N}\sum_{n=1}^{N}Z_n(y)\cdot T_1^{X_1(y)+\cdots X_n(y)}f_1\cdot T_2^{n}f_2\right\|_{L^2}^{2}\ll A_{1,N}(y)+A_{2,N}(y),\label{vdC1}
\end{equation}
     where
\[A_{1,N}(y):=N^{-1+2a-b}\sum_{n=1}^{N}\| Z_n(y)\cdot T_1^{X_1(y)+\cdots X_n(y)}f_1\cdot T_2^{n}f_2\|_{L^2}^2\]
and
\begin{equation*}
    \begin{split}
      A_{2,N}(y):=N^{-1+2a-b}\sum_{m=1}^{M}\Big|\sum_{n=1}^{N-m}\int &Z_{n+m}(y)\cdot Z_n(y)\cdot T_1^{X_1(y)+\cdots+X_{n+m}(y)}f_1\\
      &\cdot T_2^{n+m}f_2\cdot T_1^{X_1(y)+\cdots+X_n(y)}f_1\cdot T_2^nf_2\text{d}\mu\Big|.  
    \end{split}
\end{equation*}

For simplicity, we omit the dependence on $y$, and note that the constants implicit in $\ll$ depend on this parameter. 

By Proposition \ref{prop:LLN2}, noticing that $Z_n^2=X_n-2X_n\sigma_n+\sigma_n^2\leq X_n+\sigma_n^2$, we have that
\[\sum_{n=1}^{N}Z_n^2\leq \sum_{n=1}^{N}X_n+\sum_{n=1}^{N}\sigma_n^2\ll \sum_{n=1}^N\sigma_n\ll N^{1-a}\]
almost surely. Therefore
\begin{equation}
    A_{1,N}\ll N^{-1+2a-b}\sum_{n=1}^{N}Z_n^2\ll N^{-1+2a-b}N^{1-a}=N^{a-b}
\end{equation}
almost surely.

Composing with $T_2^{-n}$, using the Cauchy-Schwarz inequality and that $\|T^mf_2\cdot f_2\|_{L^2}$ is bounded, we obtain
\[A_{2,N}\ll N^{-1+2a-b}\cdot \sum_{m=1}^{M}\left \| \sum_{n=1}^{N-m} Z_{n+m}\cdot Z_n \cdot T_2^{-n}T_1^{X_1+\cdots+X_{n+m}}f_1\cdot T_2^{-n}T_1^{X_1+\cdots+X_n}f_2\right\|_{L^2}.\]
Using the triangle inequality and the fact that
\[N^{-1+2a-b}\sum_{m=1}^{M}m\ll N^{-1+2a+b},\]
it is possible to conclude
\begin{align*}
    &A_{2,N}\\
    &\ll N^{-1+2a-b}\cdot \sum_{m=1}^{M}\left \| \sum_{n=1}^{N} Z_{n+m}\cdot Z_n \cdot T_2^{-n}T^{X_1+\cdots+X_{n+m}}f_1\cdot T_2^{-n}T_1^{X_1+\cdots+X_n}f_1\right\|_{L^2}\\
    &+N^{-1+2a-b}\cdot \sum_{m=1}^{M}\left \| \sum_{n=N-m+1}^{N} Z_{n+m}\cdot Z_n \cdot T_2^{-n}T_1^{X_1+\cdots+X_{n+m}}f_1\cdot T_2^{-n}T_1^{X_1+\cdots+X_n}f_1\right\|_{L^2}\\
    &\ll N^{-1+2a-b}\cdot \sum_{m=1}^{M}\left \| \sum_{n=1}^{N} Z_{n+m}\cdot Z_n \cdot T_2^{-n}T_1^{X_1+\cdots+X_{n+m}}f_1\cdot T_2^{-n}T_1^{X_1+\cdots+X_n}f_1\right\|_{L^2}\\
    &\qquad +N^{-d_1},
\end{align*}
 where $d_1:=1-2a-b$. Using Cauchy-Schwarz, 
 \begin{align*}
     A_{2,N}^2&\ll N^{-2d_1}\\
     &+N^{-2+4a-b}\cdot \sum_{m=1}^{M}\left \| \sum_{n=1}^{N} Z_{n+m}\cdot Z_n \cdot T_2^{-n}T_1^{X_1+\cdots+X_{n+m}}f_1\cdot T_2^{-n}T_1^{X_1+\cdots+X_n}f_1\right\|_{L^2}^2,
     \end{align*}
We use Lemma~\ref{vdc} again with $R=\lfloor N^c\rfloor$ for $c:=2a+\varepsilon$,  obtaining \begin{equation}
    A_{2,N}^2\ll N^{-2d_1}+A_{3,N}+A_{4,N},\label{vdC2}
\end{equation}
where
\[A_{3,N}:=N^{-1+4a-b-c}\sum_{m=1}^{M}\sum_{n=1}^{N}\Big \| Z_{n+m}\cdot Z_n \cdot S^{-n}T^{X_1+\cdots+X_{n+m}}f\cdot S^{-n}T^{X_1+\cdots+X_n}f\Big\|^2\]
and

\begin{equation*}
\begin{aligned}
A_{4,N}:=N^{-1+4a-b-c}\cdot&\sum_{m=1}^{M}\sum_{r=1}^{R}\biggl| \sum_{n=1}^{N-r}\int Z_{n+r+m}Z_{n+r}Z_{n+m}Z_{n}\cdot S^{-n-r}T^{X_1+\cdots+X_{n+m+r}}f\\
&\cdot S^{-n-r}T^{X_1+\cdots+X_{n+r}}f\cdot S^{-n}T^{X_1+\cdots+X_{n+m}}f\cdot S^{-n}T^{X_1+\cdots+X_n}f\biggr |.
\end{aligned}
\end{equation*}

By Corollary~\ref{cor:C_1}, 
\begin{equation}
   A_{3,N}\ll N^{-1+4a-b-c}\sum_{m=1}^{M}\sum_{n=1}^{N}Z_{n+m}^2Z_n^2\ll_{\omega} N^{-1+4a-b-c}N^{b+1-2a}=N^{-\varepsilon},\label{vdC3} 
\end{equation}
almost surely. 
Composing with $T^{-(X_1+\cdots X_n)}S^{n}$ and using that $T$ and $S$ are commutative, we get
\begin{equation*}
\begin{aligned}
A_{4,N}:=N^{-1+4a-b-c}\cdot&\sum_{m=1}^{M}\sum_{r=1}^{R}\biggl|  \int f\cdot \sum_{n=1}^{N-r}Z_{n+r+m}Z_{n+r}Z_{n+m}Z_{n}\cdot \\
& S^{-r}T^{X_{n+1}+\cdots+X_{n+m+r}}f\cdot S^{-r}T^{X_{n+1}+\cdots+X_{n+r}}f\cdot T^{X_{n+1}+\cdots+X_{n+m}}f d\mu\biggr |.
\end{aligned}
\end{equation*}

By the Cauchy-Schwarz inequality and recalling $f\in L^{\infty}(\mu)$, 
\begin{align*}
    A_{4,N}\ll & N^{-1+4a-b-c}\cdot \sum_{m=1}^{N^b}\sum_{r=1}^{N^c}\biggl \|\sum_{n=1}^{N-r}Z_{n+m+r}\cdot Z_{n+r}\cdot Z_{n+m}\cdot Z_n\cdot\\
    &T^{X_{n+1}+\cdots+X_{n+m+r}}S^{-r}f\cdot T^{X_{n+1}+\cdots+X_{n+r}}S^{-r}f\cdot T^{X_{n+1}+\cdots X_{n+m}}f\biggr \|_{L^2}
\end{align*}

Note that, for any $k\in\N$, $X_{n+1}+\cdots X_{n+k}\in \left\{0,...,k \right\}$. Hence, for every $y\in [0,1]$ and $h\in L^{\infty}(\mu)$,
\[T^{X_{n+1}(y)+\cdots+X_k(y)}h=\sum_{j=0}^{k} \mathbbm{1}_{\sum_{i=1}^{m+r}X_{n+i}(y)=j}T^{j}h.\]
After applying this idea and the triangular inequality we obtain

\begin{align*} A_{4,N}\ll &N^{-1+4a-b-c}\cdot\sum_{m=1}^{M}\sum_{r=1}^{R}\sum_{k_1=0}^{m+r}\sum_{k_2=0}^{r}\sum_{k_3=0}^m\biggl \| \sum_{n=1}^{N-r}Z_{n+m+r}\cdot Z_{n+r}\cdot Z_{n+m}\cdot Z_n\cdot \\
&\mathbbm{1}_{\sum_{i=1}^{m+r}X_{n+i}=k_1}(n)S^{-r}T^{k_1}f\cdot \mathbbm{1}_{\sum_{i=1}^{r}X_{n+i}=k_2}(n)S^{-r}T^{k_2}f\cdot \mathbbm{1}_{\sum_{i=1}^{m}X_{n+i}=k_3}(n)T^{k_3}f\biggr \|_{L^2}.
\end{align*}
Note that the random variables are viewed as constant inside the $L^2(X,\mu)$ norm, then

\begin{align*} A_{4,N}\ll A_{5,N}:= N^{-1+4a-b-c}\cdot&\sum_{m=1}^{M}\sum_{r=1}^{R}\sum_{k_1=0}^{m+r}\sum_{k_2=0}^{r}\sum_{k_3=0}^m\biggl | \sum_{n=1}^{N-r}Z_{n+m+r}\cdot Z_{n+r}\cdot Z_{n+m}\cdot Z_n\cdot \\
&\mathbbm{1}_{\sum_{i=1}^{m+r}X_{n+i}=k_1}(n)\cdot \mathbbm{1}_{\sum_{i=1}^{r}X_{n+i}=k_2}(n)\cdot \mathbbm{1}_{\sum_{i=1}^{m}X_{n+i}=k_3}(n)\biggr |.
\end{align*}

Using  equations \eqref{vdC1},\eqref{vdC2}, \eqref{vdC3} and the previous estimate, for almost every $y\in [0,1]$,
\begin{equation}
    A_{N}^2(y)\ll_{y}N^{2(a-b)}+N^{-2d_1}+N^{-\varepsilon}+A_{5,N}(y).\label{eq1_dem}
\end{equation}
Note that at this point, the upper bound for $A_N$ no longer depends on the measure-preserving system; it depends only on the random variables. This is essential to ensure we can choose a full-measure set that works for every m.p.s. 

Let us note that
\begin{equation*}
\begin{split}
    \mathbb{E}(A_{5,N})\ll N^{-1+4a+c}\sum_{m=1}^{M}\sum_{r=1}^{R}\mathbb{E}\Bigg(\dfrac{1}{m(m+r)r}\sum_{k_1=0}^{m+r}\sum_{k_2=0}^{r}\sum_{k_3=0}^{m}\biggl | \sum_{n=1}^{N-r} Z_n\cdot  U_{n,m,r,k_1,k_2,k_3}\biggr|\Bigg)
\end{split}\label{A5N}
\end{equation*} where
\[U_{n,m,r,k_1,k_2,k_3}:=Z_{n+m+r}\cdot Z_{n+r}\cdot Z_{n+m} \cdot \mathbbm{1}_{\sum_{i=1}^{m+r}X_{n+i}=k_1}\mathbbm{1}_{\sum_{i=1}^{r}X_{n+i}=k_2}\mathbbm{1}_{\sum_{i=1}^{m}X_{n+i}=k_3}.\]
Using Jensen's inequality,

\begin{equation*}
    \begin{split}
        &\Bigg(\mathbb{E}_{\omega}\Bigg(\dfrac{1}{m(m+r)r}\sum_{k_1=0}^{m+r}\sum_{k_2=0}^{r}\sum_{k_3=0}^{m}\biggl | \sum_{n=1}^{N-r} Z_n\cdot  U_{n,m,r,k_1,k_2,k_3}\biggr|\Bigg)\Bigg)^2\\
        &\ll \dfrac{1}{m(m+r)r}\mathbb{E}\Bigg(\sum_{k_1=0}^{m+r}\sum_{k_2=0}^{r}\sum_{k_3=0}^{m}\biggl | \sum_{n=1}^{N-r} Z_n\cdot  U_{n,m,r,k_1,k_2,k_3}\biggr|^2\Bigg)\\
        &\ll \dfrac{1}{m(m+r)r}\sum_{k_1=0}^{m+r}\sum_{k_2=0}^{r}\sum_{k_3=0}^{m}\sum_{n=1}^{N} \mathbb{E}\big(Z_n^2 \cdot U_{n,m,r,k_1,k_2,k_3}^2\big)\\
        &\quad +\dfrac{1}{m(m+r)r}\sum_{n_1\neq n_2}^{N-r}\mathbb{E}\Bigg(\sum_{k_1=0}^{m+r}\sum_{k_2=0}^{r}\sum_{k_3=0}^{m} Z_{n_1}  U_{n_1,m,r,k_1,k_2,k_3}\cdot Z_{n_2}  U_{n_2,m,r,k_1,k_2,k_3}\Bigg)\\
        &\ll \sum_{n=1}^N\mathbb{E}\Big(Z_n^2Z_{n+m}^2Z_{n+r}^2Z_{n+m+r}^2\Big)+\dfrac{1}{m(m+r)r}\sum_{n_1=1}^{N-r}\sum_{n_2=n_1+1}^{N-r} \mathbb{E}\Big(Z_{n_1}\widetilde{Z}_{n_1,n_2,m,r} \widetilde{U}_{n_1,n_2,m,r}\Big),
    \end{split}
\end{equation*}
where 
\begin{align*}
    \widetilde{Z}_{n_1,n_2,m,r}&:=Z_{n_1+m}Z_{n_1+r}Z_{n_1+m+r}Z_{n_2}Z_{n_2+m}Z_{n_2+r}Z_{n_2+m+r},\\
\widetilde{U}_{n_1,n_2,m,r}&:=\mathbbm{1}_{\sum_{i=1}^{m+r}X_{n_1+i}=\sum_{i=1}^{m+r}X_{n_2+i}}\mathbbm{1}_{\sum_{i=1}^{r}X_{n_1+i}=\sum_{i=1}^{r}X_{n_2+i}}\mathbbm{1}_{\sum_{i=1}^{m}X_{n_1+i}=\sum_{i=1}^{m}X_{n_2+i}}.
\end{align*}
Therefore,
\begin{equation}
    \begin{split}
        \mathbb{E}(A_{5,N})&\ll N^{-1+4a+c}\sum_{m=1}^{M}\sum_{r=1}^{R}\sqrt{ \sum_{n=1}^N\mathbb{E}\Big(Z_n^2Z_{n+m}^2Z_{n+r}^2Z_{n+m+r}^2\Big)}\\
        &+N^{-1+4a+c}\sum_{m=1}^{M}\sum_{r=1}^{R}\dfrac{1}{r}\sqrt{\sum_{n_1=1}^{N-r}\sum_{n_2=n_1+1}^{N-r} \Big|\mathbb{E}\Big(Z_{n_1}\widetilde{Z}_{n_1,n_2,m,r} \widetilde{U}_{n_1,n_2,m,r}\Big)\Big|}.
    \end{split}\label{A_5}
\end{equation}
\\

By expanding the squares and rearranging terms, we can write
\begin{equation}
    \begin{split}
        &\mathbb{E}(Z_n^2Z_{n+m}^2Z_{n+r}^2Z_{n+m+r}^2)\\
    =&\operatorname{Cov}(Z_n^2,Z_{n+m}^2Z_{n+r}^2Z_{n+m+r}^2)+\mathbb{E}(Z_n^2)\operatorname{Cov}(Z_{n+m}^2,Z_{n+r}^2Z_{n+m+r}^2)\\
    &+\mathbb{E}(Z_n^2)\mathbb{E}(Z_{n+m}^{2})\Big(\operatorname{Cov}(Z_{n+r}^2,Z_{n+m+r}^2)+\mathbb{E}(Z_{n+r}^2)\mathbb{E}(Z_{n+m+r}^{2})\Big).
    \end{split}\label{thmC:decomposition}
\end{equation}

If we write $\displaystyle h_k:=\mathbbm{1}_{I_k}-\sigma_k,$ note that
$Z_n^2=h_n^2\circ S^{\varphi(n)}$ \text{ and }$$Z_{n+m}^2Z_{n+r}^2Z_{n+m+r}^2=h_{n+m}^2\circ S^{\varphi(n+m)}\cdot h_{n+r}^2\circ S^{\varphi(n+r)}\cdot h_{n+m+r}^2\circ S^{\varphi(n+m+r)}.$$
Note that $\|h_k\|_{\operatorname{BV}}\ll 1$. Using that the system $([0,1],\mathcal{B}([0,1]),\nu,S)$ has a decay of correlation with exponential rate $\rho$, we obtain that
$$\Big|\operatorname{Cov}(Z_n^2,Z_{n+m}^2Z_{n+r}^2Z_{n+m+r}^2)\Big|\ll \rho\big( \varphi(n+\min\{m,n\})-\varphi(n)\big )\ll  \rho\big( \varphi(n+1)-\varphi(n)\big ).$$
If $r>m$, we can proceed analogously for the other terms of (\ref{thmC:decomposition}) and obtain
\begin{align*}
     \mathbb{E}(Z_n^2Z_{n+m}^2Z_{n+r}^2Z_{n+m+r}^2)\ll \rho(\varphi(n+1)-\varphi(n))+\sigma_{n}^4.
\end{align*}
In case $m\geq r$, we use the trivial estimation $\displaystyle |\mathbb{E}(Z_n^2)\operatorname{Cov}(Z_{n+m}^2,Z_{n+r}^2Z_{n+m+r}^2)|\ll \sigma_n^2.$ 
Therefore,

\begin{equation}
    \begin{split}
        &\sum_{m=1}^{M}\sum_{r=1}^{R}\sqrt{ \sum_{n=1}^N\mathbb{E}\Big(Z_n^2Z_{n+m}^2Z_{n+r}^2Z_{n+m+r}^2\Big)}\\
        &\ll  \sum_{m=1}^{M}\sum_{r=1}^{R}\sqrt{\sum_{n=1}^N \rho(\varphi(n+1)-\varphi(n))+\sigma_{n}^4}+\sum_{m=1}^{M}\sum_{r=1}^{m}\sqrt{\sum_{n=1}^N\sigma_n^2}\\
        &\ll N^{b+c}\log^{1/2}(N)+N^{b+c+1/2-2a}+N^{2b+1/2-a}\\
        &\ll N^{1/2+b+c-2a}+N^{1/2+2b-a}.
    \end{split}\label{A5_1}
\end{equation}
On the other hand, using the decay of correlation, we have that
\begin{equation*}
\begin{split}
    \Big|\mathbb{E}\Big(Z_{n_1}\widetilde{Z}_{n_1,n_2,m,r} \widetilde{U}_{n_1,n_2,m,r}\Big)\Big|&=\Big|\operatorname{Cov}(X_{n_1},\widetilde{Z}_{n_1,n_2,m,r} \widetilde{U}_{n_1,n_2,m,r})\Big|\\
    &\ll \rho \Big(\varphi(n_1+1)-\varphi(n_1)\Big).
\end{split}
\end{equation*}
Therefore,
\begin{equation}
    \begin{split}
        \sum_{m=1}^{M}\sum_{r=1}^{R}\dfrac{1}{r}\sqrt{\sum_{n_1=1}^{N-r}\sum_{n_2=n_1+1}^{N-r} \Big|\mathbb{E}\Big(Z_{n_1}\widetilde{Z}_{n_1,n_2,m,r} \widetilde{U}_{n_1,n_2,m,r}\Big)\Big|}\ll N^b\log(N)^{3/2}N^{1/2}.
    \end{split}\label{A5_2}
\end{equation}
Using (\ref{A5_1}) and (\ref{A5_2}) in (\ref{A_5}),
\begin{equation}
    \begin{split}
    \mathbb{E}(A_{5,N})&\ll N^{-1+4a+c}\Big( N^{1/2+b+c-2a}+N^{1/2+2b-a}+N^b\log(N)^{3/2}N^{1/2}\Big)\\
    &\ll N^{-1/2+4a+c+b}\Big(N^{c-2a}+N^{b-a}+\log(N)^{3/2}\Big).
    \end{split}
\end{equation}
Therefore, recalling $c=2a+\varepsilon$ and $b=a+\varepsilon$,
$$\mathbb{E}(A_{5,N})\ll N^{-d_2},$$
where $d_2:=1/2-7a-3\varepsilon>0.$ In particular, for any $\gamma>1$,
$$\sum_{N=1}^\infty \mathbb{E}(A_{5,\lfloor \gamma^N\rfloor})<\infty.$$
Hence, for almost every $y\in [0,1]$,
$$\sum_{N=1}^\infty A_{5,\lfloor \gamma^N\rfloor}(y)<\infty.$$
It follows from (\ref{eq1_dem}) that, for almost every $y\in [0,1]$ (independent of the measure-preserving system),
\begin{equation*}
       \sum_{N=1}^\infty A_{\lfloor \gamma^N\rfloor}^2(y) <\infty,
\end{equation*}
finishing the proof.
 \end{proof}

\section{Proof of Theorem  \ref{thmD}} \label{sec:ProofD}
\begin{proof}
Let us define $Z_n := X_n - n^{-a}$. Since the averages
\[
\frac{1}{N}\sum_{n=1}^N T^{n}f \cdot T^{\lfloor n^{1+b}\rfloor} g
\]
converge in $L^2$ (see \cite{BMR}), it follows, as in the previous cases, that a standard lacunary
argument reduces the problem to showing that for any $\rho>1$, almost surely, the following holds:

For any measure preserving system $(X,\mathcal{X},\mu,T)$ and for any $f,g\in L^\infty(\mu)$,
\begin{equation}
\sum_{N\in \left\{\lfloor\rho^k\rfloor:k\in \setN\right\}} \left\|\frac{1}{N^{1-a}}\sum_{n=1}^N Z_n (\omega)T^{n}f\cdot T^{\lfloor n^{1+b}\rfloor}g \right\|^2_{L^2(\mu)}<\infty.
\end{equation}
By applying van der Corput's lemma with $M=\lfloor N^c \rfloor$, $c<1$, we get
\begin{equation}
A_N^2:= \left\|\frac{1}{N^{1-a}}\sum_{n=1}^N Z_n T^{n}f\cdot T^{\lfloor n^{1+b}\rfloor}g \right\|^2\ll  A_{1,N}+A_{2,N},
\end{equation}
where 
\begin{align*}
A_{1,N}&= \frac{N^{1-c}}{N^{2-2a}}\sum_{n=1}^N \left\|Z_n\cdot T^n f \cdot T^{\lfloor n^{1+b}\rfloor}g \right\|_{L^2(\mu)}^2 \text{  and }\\
A_{2,N}&= \frac{N^{1-c}}{N^{2-2a}}\sum_{m=1}^{\lfloor N^{c}\rfloor} \left|\int_{X}\sum_{n=1}^{N-m} Z_n Z_{n+m} T^{n}f \cdot T^{\lfloor n^{1+b}\rfloor}g \cdot T^{n+m} \overline{f} \cdot T^{\lfloor {(n+m)}^{1+b}\rfloor} \overline{g}  d\mu\right|.
\end{align*}

Recalling that the random variables are independent, by the strong law of large numbers and using the fact that $\mathbb{E}(X_n^2)=n^{-a}$, almost surely we have
\begin{align*}
A_{1,N} \leq  \frac{N^{1-c}}{N^{2-2a}}\sum_{n=1}^N |Z_n|^2 \ll \frac{N^{1-c}}{N^{2-2a}}\cdot N^{1-a}=N^{a-c}.
\end{align*}
Composing with $T^{-n}$ and using the Cauchy-Schwarz inequality, we get 
\begin{align*}
A_{2,N} &= \frac{N^{1-c}}{N^{2-2a}}\sum_{m=1}^{\lfloor N^{c}\rfloor} \left|\int_{X}\sum_{n=1}^{N-m} Z_n Z_{n+m} f\cdot  T^{\lfloor n^{1+b}-n\rfloor}g \cdot T^{m} \overline{f} \cdot T^{\lfloor {(n+m)}^{1+b}-n\rfloor} \overline{g}  d\mu\right|\\
&\leq \frac{N^{1-c}}{N^{2-
2a}}\sum_{m=1}^{\lfloor N^{c}\rfloor} \left\|\sum_{n=1}^{N-m} Z_n Z_{n+m}  T^{\lfloor n^{1+b}-n\rfloor}g \cdot T^{\lfloor {(n+m)}^{1+b}-n\rfloor} \overline{g} \right\|_{L^2(\mu)}
\end{align*}
Observe that if $0<b<1$, $1\leq m \leq N^c$  and  $1\leq n \leq N$  then we have \begin{equation}
0\leq (n+m)^{1+b}-n^{1+b}\leq N^{(b+c)}.
\end{equation}

Hence, for every $m\in [\lfloor N^c\rfloor]$, we can partition the set $[N]$ as $$[N]= \bigsqcup_{r=1 }^{\lfloor N^{(b+c)}\rfloor} S_{m,r},$$ where $$S_{m,r}=\left\{n\in [N]: \lfloor {(n+m)}^{1+b}-n\rfloor- \lfloor n^{1+b}-n\rfloor=r \right\}.$$ Thus, we get
\begin{align*}
A_{2,N}&\leq \frac{N^{1-c}}{N^{2-2a}}\sum_{m=1}^{\lfloor N^{c}\rfloor} \sum_{r=1}^{\lfloor N^{(b+c)}\rfloor} \left\|\sum_{n\in [N-m]\cap S_{m,r}} Z_n Z_{n+m}  T^{\lfloor n^{1+b}-n\rfloor}g \cdot T^{(\lfloor {n}^{1+b}-n\rfloor+r)}\overline{g} \right\|_{L^2(\mu)}\\
\end{align*}
Letting $g_r=g\cdot T^r \overline{g}$, we get
\begin{align}\label{eq:bound}
A_{2,N}&\leq \frac{N^{1-c}}{N^{2-2a}}\sum_{m=1}^{\lfloor N^{c}\rfloor} \sum_{r=1}^{\lfloor N^{(b+c)}\rfloor} \left\|\sum_{n\in [N-m]\cap S_{m,r}} Z_n Z_{n+m}  T^{\lfloor n^{1+b}-n\rfloor}g_r \right\|_{L^2(\mu)}.
\end{align}
Using Herglotz's spectral theorem and the fact that $g_r$ is bounded by $1$, we get
\begin{equation}\label{eq:spectral}
\begin{aligned}
 &\left\|\sum_{n\in [N-m]\cap S_{m,r}} Z_n Z_{n+m}  T^{\lfloor n^{1+b}-n\rfloor}g_r \right\|_{L^2(\mu)} \\
 &\leq \max_{m\in \lfloor N^c\rfloor} \max_{r\in [N^{b+c}]} \sup_{t\in [0,1]} \left|\sum_{n\in [N-m]\cap S_{m,r}} Y_n Y_{n+m}  e\left({\lfloor n^{1+b}-n\rfloor}t\right) \right|
\end{aligned}
\end{equation}

For fixed $r\in [N^{b+c}]$ and $m\in [N^c]$, we want to apply Lemma \ref{lem:technical} for $Z_{m,n}=Z_n \cdot Z_{n+m}$, $s(n)={\lfloor n^{1+b}-n\rfloor}$ and $\Lambda= S_{m,r}\cap [N-m]\cap \Lambda_{k,m}$, $k=1,2$ where $\Lambda_{k,m}$ are defined as follows:
\begin{equation*}
\begin{aligned}
\Lambda_{1,m}&=\{n: 2km<n\leq (2k+1)m \text{ for some non-negative integer }k\}, \text{ and }\\
\Lambda_{2,m}&=\{n: (2k+1)m<n\leq (2k+2)m \text{ for some non-negative integer }k\}.
\end{aligned}
\end{equation*}
Observe that $\Lambda_{1,m}\cup\Lambda_{2,m}=\setN$, $\{Z_{m,n}: n\in \Lambda_{k,m}\}$ is independent, and 
\begin{align*}
\text{Var}(Z_{m,n})&= \sigma_n \sigma_{m+n}-\sigma_n\sigma_{n+m}^2-\sigma_n^2\sigma_{n+m}+ \sigma_n^2 \sigma_{n+m}^2\\
&\leq \sigma_n\sigma_{n+m}\leq \sigma_n^2\sim n^{-2a}.
\end{align*}
Thus, condition \eqref{eq:condition} is satisfied with $\rho_n= n^{-2a}$. By Lemma~\ref{lem:technical} we get 

\begin{equation*}
\mathbb{P}\left(\sup_{t\in [0,1]} \left|\sum_{n\in [N-m]\cap S_{m,r}} Z_n Z_{n+m}  e\left({\lfloor n^{1+b}-n\rfloor}t\right) \right| \geq 2A\sqrt{ R_N \log N} \right) \leq  \frac{2}{N^4}.
\end{equation*} 
Hence,
\begin{align*}
&\mathbb{P}\left(\max_{m\in \lfloor N^c\rfloor} \max_{r\in [N^{b+c}]}\sup_{t\in [0,1]} \left|\sum_{n\in [N-m]\cap S_{m,r}} Z_n Z_{n+m}  e\left({\lfloor n^{1+b}-n\rfloor}t\right) \right| \geq 2A\sqrt{ R_N \log N} \right)\\&=
\mathbb{P}\left(\bigcup_{m\in \lfloor N^c\rfloor} \bigcup_{r\in [N^{b+c}]}\sup_{t\in [0,1]} \left|\sum_{n\in [N-m]\cap S_{m,r}} Z_n Z_{n+m}  e\left({\lfloor n^{1+b}-n\rfloor}t\right) \right| \geq 2A\sqrt{ R_N \log N} \right)\\
&\leq \sum_{m\in [N^c]} \sum_{r\in [N^{b+c}]} \mathbb{P}\left(\sup_{t\in [0,1]} \left|\sum_{n\in [N-m]\cap S_{m,r}} Z_n Z_{n+m}  e\left({\lfloor n^{1+b}-n\rfloor}t\right) \right| \geq 2A\sqrt{ R_N \log N} \right)\\
&\leq \frac{2N^{b+2c}}{N^4}\\
&\leq \frac{2}{N^2}.
\end{align*}
By the Borel-Cantelli lemma, almost surely we have
\begin{equation*}
\max_{m\in  [N^c]} \max_{r\in [N^{b+c}]} \sup_{t\in [0,1]} \left| \sum_{[N]\cap J_{m,r}} Z_{m,n}\cdot e(s(n)t)\right| \ll_{\omega} \left(\sqrt{\log N\cdot \sum_{n=1}^N \rho_n }\right).
\end{equation*}

Combining \eqref{eq:bound} and \eqref{eq:spectral} we get,

\begin{align*}
A_{2,N} &\ll_{\omega} \frac{N^{1-c}}{N^{2-2a}}\sum_{m=1}^{\lfloor N^{c}\rfloor} \sum_{r=1}^{\lfloor N^{(b+c)}\rfloor} \sqrt{\log N} \cdot N^{\frac{1}{2}-a}\\ &\leq N^{1-c-2+2a+c+b+c+\frac{1}{2}-a}= N^{a+c+b-\frac{1}{2}}.
\end{align*}
In conclusion, almost surely $(A_N^2)$ is summable along any lacunary sequence as long as we have
\begin{equation}\label{eq:restriction}
b<1, a<c<1 \text{ and } b+c+a<\frac{1}{2}.
\end{equation}
Since we are given $(b+2a)<\frac{1}{2}$, it is possible to choose $c$ satisfying \eqref{eq:restriction}. This completes the proof of the theorem.
\end{proof}

\bibliographystyle{abbrv}
\bibliography{refs,ref}
\end{document}